\documentclass[a4paper,10pt,reqno]{amsart}

\usepackage[english]{babel}
\usepackage[utf8]{inputenc}
\usepackage{amsmath,amssymb,amsfonts,amsthm,amscd}
\usepackage[mathscr]{eucal}
\usepackage{mathtools}

\usepackage{enumitem}
\setlist[itemize]{topsep=0.2em, itemsep=0.2em, leftmargin=2em}
\setlist[enumerate]{topsep=0.2em, itemsep=0.2em, leftmargin=2em}
\setlist[description]{topsep=0.2em, itemsep=0.2em, leftmargin=2em}

\usepackage{tikz-cd}
\usepackage{hyperref}

\newcommand{\R}{\mathbb{R}}
\newcommand{\C}{\mathbb{C}}

\newcommand{\df}{\mathrm{d}}

\newcommand{\M}{\mathbb{M}}

\DeclareMathOperator{\arctanh}{arctanh}
\DeclareMathOperator{\sech}{sech}

\DeclareMathOperator{\am}{am}
\DeclareMathOperator{\Rot}{Rot}
\DeclareMathOperator{\Area}{Area}
\DeclareMathOperator{\Vol}{Vol}
\newcommand{\E}{\mathbb{E}}
\newcommand{\X}{\mathfrak{X}}
\newcommand{\id}{\mathrm{id}}

\definecolor{resaltar}{rgb}{0.8,0,0.3}
\newcommand{\alert}[1]{%
  \textcolor{resaltar}{#1}
}

\newcommand{\alertm}[1]{%
  \marginpar{%
    \ifodd\value{page} \raggedright\else \raggedleft\fi
    \footnotesize{\alert{#1}}
  }
}

\newtheorem{theorem}{Theorem}[section]

\newtheorem{corollary}[theorem]{Corollary}
\newtheorem{lemma}[theorem]{Lemma}

\theoremstyle{definition}

\theoremstyle{remark}
\newtheorem{remark}[theorem]{Remark}

\numberwithin{equation}{section}

\title[Invariant $H$-tubes around a horizontal geodesic in $\E(\kappa,\tau)$-spaces]{Invariant constant mean curvature tubes around a horizontal geodesic in $\E(\kappa,\tau)$-spaces}
\date{}

\author{Jos\'e M.Manzano}
\address{Departamento de Matem\'aticas, Universidad de Ja\'en, 23071 Ja\'en SPAIN}
\email{jmprego@ujaen.es}

\thanks{}



\begin{document}

\begin{abstract}
We consider constant mean curvature surfaces (invariant by a continuous group of isometries) lying at bounded distance from a horizontal geodesic on any homogeneous $3$-manifold $\mathbb{E}(\kappa,\tau)$ with isometry group of dimension $4$. These surfaces are called horizontal tubes. We show that they foliate $\E(\kappa,\tau)$ minus one or two horizontal geodesics provided that $(1-x_0^2)\kappa+4\tau^2\leq 0$, where $x_0\approx 0.833557$. We also describe precisely how horizontal and vertical geodesics get deformed by Daniel's sister correspondence and conclude that the family of horizontal tubes is preserved by the correspondence. These tubes are topologically tori in $\mathbb{S}^2\times\mathbb{R}$ and Berger spheres, in which case we compute their conformal type and analyze numerically their isoperimetric profiles.  
\end{abstract}

\maketitle

\section{Introduction}

The study of constant mean curvature surfaces ($H$-surfaces in the sequel) in homogeneous $3$-manifolds has been particularly relevant during the last decades. This involves a twofold interest: on the one hand, to discover up to which extent properties that were well known in the case of Euclidean space $\R^3$ (and other space forms) still hold under the homogeneity condition; on the other hand, to find novel features which cannot be found in the classical theory. In the latter case, the construction of examples has played a key role, starting with the simplest case of $H$-surfaces invariant by $1$-parameter groups of isometries (e.g., see~\cite{FMP,Kase,Man13,Onnis,Pena,Tor10} and the references therein) because these surfaces often possess characteristics that throw light on the behavior of more sophisticated non-invariant examples.

We will focus on a family of invariant $H$-surfaces in $\E(\kappa,\tau)$ that has received little attention, likely because they are subtle to spot. These surfaces show up in the supercritical range of the mean curvature (i.e., if $4H^2+\kappa>0$) and will be called \emph{horizontal $H$-tubes} because they are invariant under translations along a horizontal geodesic. If $\kappa-4\tau^2=0$, they reduce to right $H$-cylinders in $\R^3$ or Hopf $H$-tori in the round sphere $\mathbb{S}^3$. The $H$-tubes have been already found in $\mathbb{S}^2\times\R$ by Pedrosa~\cite{Pedrosa}, in $\mathbb{H}^2\times\mathbb{R}$ by Onnis~\cite{Onnis} and Vržina~\cite{Vrzina} (including also the special linear group $\E(\kappa,\tau)$, $\kappa\leq 0$), in Heisenberg group $\mathrm{Nil}_3(\tau)=\E(0,\tau)$ by Figueroa, Mercuri and Pedrosa~\cite{FMP}. Very recently, Käse~\cite{Kase} has completed the picture by classifying $H$-surfaces in $\E(\kappa,\tau)$-spaces invariant with respect to circular screw-motions, among which he gets the $H$-tubes for a particular choice of the pitch. Note that there was already a well known notion of \emph{vertical $H$-tubes}, defined as the preimages by the Riemannian submersion $\pi:\E(\kappa,\tau)\to\mathbb{M}^2(\kappa)$ of curves of constant curvature $2H$, see Section~\ref{sec:preliminaries}. It seems likely that our construction can be extended to any geodesic of $\E(\kappa,\tau)$, since there is numerical evidence that $H$-tubes exist for $H>H_0$ for some $H_0>0$ depending on $\kappa$, $\tau$, and the prescribed geodesic. The product case has been further studied in~\cite{MT22} as part of the families of horizontal Delaunay $H$-surfaces. It would be also interesting to obtain $H$-surfaces of Delaunay type that bifurcate the rest of horizontal $H$-tubes with nonzero $\tau$.

In this work, we are interested in the global geometry of $H$-tubes about a given horizontal geodesic $\Gamma\subset\E(\kappa,\tau)$. In Section~\ref{sec:invariant}, we will present a common framework that gives unified expressions valid for all $\kappa$ and $\tau$, despite of the fact that $H$-tubes are invariant under circular screw motions ($\kappa>0$), translations ($\kappa=0$) or hyperbolic translations ($\kappa<0$). We will see that $H$-tubes are spanned by convex curves in the orbit space, which will help us show that $H$-tubes foliate $\E(\kappa,\tau)-\Gamma$ if $\kappa\leq 0$. As a consequence, we show that there is no immersed $H$-surface with $4H^2+\kappa\leq 0$ at bounded distance from a horizontal geodesic as an application of Mazet's halfspace theorem~\cite{Mazet2013}. If $\kappa>0$, then the $H$-tubes around $\Gamma$ produce a foliation of $\E(\kappa,\tau)-(\Gamma\cup\Gamma')$ if $(1-x_0^2)\kappa-4\tau^2\leq 0$, where $x_0\approx 0.833557$ and $\Gamma$ and $\Gamma'$ differ in a vertical translation, see Theorem~\ref{thm:foliation}. As a matter of fact, such a foliation becomes a foliation by Hopf $H$-tori in the case of the round sphere $(\kappa,\tau)=(4,1)$, in which case horizontal and vertical $H$-tubes are congruent. However, the geometry of the horizontal $H$-tubes is considerably more involved in general (e.g., vertical $H$-tubes are embedded and flat, which is not true in the horizontal case).

In the minimal case (which is supercritical only if $\kappa>0$), horizontal $H$-tubes reduce to a horizontal slices in $\mathbb{S}^2\times\R$ and to particular spherical helicoids in Berger spheres (which are Clifford tori for the round metric), see Remark~\ref{rmk:clifford}. The fact that these minimal tori are conjugate to horizontal $H$-tubes in $\mathbb{S}^2\times\R$ and $\mathbb{H}^2\times\R$ (see~\cite[\S4]{MT14}) has motivated us to investigate the whole associate family of sister immersions of minimal $H$-tori, in the sense of Daniel~\cite{Dan}. This family depends on a phase angle $\theta\in\R$ such that the value $\theta=\frac{\pi}{2}$ gives the aforementioned conjugation. In independent works, Plehnert~\cite{Ple14} and Torralbo and the author~\cite{MT14} established the foundations of conjugate constructions starting with a minimal surface to produce $H$-surfaces in product spaces (we recommend the recent survey~\cite{CMT} for further reference on this topic). However, the case $\theta\neq\frac{\pi}{2}$ has not been treated in the literature so far. Lemmas~\ref{lem:vertical-geodesics} and~\ref{lem:horizontal-geodesics} give a fairly complete depiction of the deformation of vertical and horizontal geodesics in terms of the angle of rotation of the normal. Although this analysis might be somewhat disappointing (the intermediate curves with $0<\theta<\tfrac{\pi}{2}$ do not seem to enjoy remarkable properties), the given formulas can be used to give precise numerical approximations of conjugate curves (since the angle of rotation can be easily calculated given a discrete approximation of the surface). This will be treated in a future work, see also~\cite[\S7]{CMT}.

Our study of sister surfaces will show that $H$-tubes are preserved by Daniel's correspondence (see Theorem~\ref{thm:spherical-helicoids}), which is a bit of a surprise because it means that all these $H$-surfaces close their periods simultaneously. (Usually, it is difficult to prove that a period closes.) We shall see that the correspondence induces an isometry between the universal covers of sister $H$-tubes, though each of them corresponds to a particular quotient, which can be seen as choosing a particular lattice in $\R^2$. We will give a precise description of the conformal class of these examples and show that all conformal classes of genus-one compact Riemann surfaces can be realized in this way (see Remark~\ref{rmk:conformal}). The case of vertical Hopf $H$-tori in Berger spheres has been studied by Torralbo and Urbano~\cite{TU}, which generalizes some of the results of Pinkall in the round sphere~\cite{Pinkall}.

In~\cite{TU}, it is also proved that in some Berger spheres the solutions to the isoperimetric problem (i.e., the least area surface spanning a given volume) are either $H$-spheres or embedded $H$-tori. Indeed, they prove that $H$-spheres are the solutions in some range of $\kappa$ and $\tau$ for all prescribed volumes. Therefore, the horizontal $H$-tubes become natural competitors to solve the isoperimetric problem. Although it seems impossible to compare explicitly areas and volumes, we will show numerically that the $H$-tubes have larger area than $H$-spheres if $\kappa-4\tau^2>0$ and have larger area than Hopf $H$-tori is $\kappa-4\tau^2<0$ (see Figure~\ref{fig:profiles}). This suggests that they cannot be solutions of the isoperimetric problem.

\medskip
\noindent\textbf{Acknowledgement.} I would like to thank Prof.\ Torralbo for suggesting the calculation of the isoperimetric profiles of the horizontal $H$-tori in Berger spheres. This research has been supported by the Ramón y Cajal fellowship RYC2019-027658-I as well as by the project PID2019.111531GA.I00, both funded by the Spanish agency MCIN/AEI/10.13039/501100011033. It has been also supported by the FEDER-UJA project No.\ 1380860.

\section{Preliminaries}\label{sec:preliminaries}

Given $\kappa,\tau\in\R$, there is a unique simply connected oriented Riemannian $3$-manifold $\E(\kappa,\tau)$ that admits a unit Killing submersion $\pi:\E(\kappa,\tau)\to\mathbb{M}^2(\kappa)$ with constant bundle curvature $\tau$ over the complete base surface $\mathbb{M}^2(\kappa)$ with constant curvature $\kappa$, see~\cite{Dan,Man14}. This means that the fibers of $\pi$ are the integral curves of a unit Killing vector field $\xi$ in $\E(\kappa,\tau)$. It turns out that the family $\E(\kappa,\tau)$ contains all the homogeneous $3$-manifolds with isometry group of dimension $4$ when $\kappa-4\tau^2\neq 0$. It also contains the space forms of non-negative constant sectional curvature, which can be recovered as $\mathbb{M}^3(c)=\E(4c,\sqrt{c})$ for any $c\geq 0$.

There is a standard local model that covers all $\E(\kappa,\tau)$-spaces. It was originally introduced by Cartan and consists in the open subset of $\R^3$ defined as
\[M(\kappa,\tau)=\{(x,y,z)\in\R^3:\lambda_\kappa(x,y)>0\}, \text{ where } \lambda_\kappa(x,y)=\left(1+\tfrac{\kappa}{4}(x^2+y^2)\right)^{-1},\]
endowed with the Riemannian metric
\[\mathrm{d} s^2=\lambda_\kappa^2(\mathrm{d} x^2+\mathrm{d} y^2)+(\mathrm{d} z+\tau\lambda_\kappa(y\mathrm{d} x-x\mathrm{d} y))^2.\]
The orientation in $M(\kappa,\tau)$ is chosen such that
\begin{equation}\label{eqn:standard-frame}
\begin{aligned}
 E_1&=\tfrac{1}{\lambda_\kappa}\partial_x-\tau y\,\partial_z,& 
 E_2&=\tfrac{1}{\lambda_\kappa}\partial_y+\tau x\,\partial_z,&
 E_3&=\partial_z
\end{aligned}\end{equation}
is a global positively oriented orthonormal frame. It follows that $\pi(x,y,z)=(x,y)$ is a Killing submersion with constant bundle curvature $\tau$ and unit Killing vector field $\xi=E_3$, where the Riemannian metric $\mathrm{d}s_\kappa^2=\lambda_\kappa^2(\mathrm{d}x^2+\mathrm{d}y^2)$ with constant curvature $\kappa$ is considered in the projection. Therefore, $M(\kappa,\tau)$ is a global model of $\E(\kappa,\tau)$ if and only if
 $\kappa\leq 0$. The Levi-Civita connection in the frame~\eqref{eqn:standard-frame} reads
\begin{equation}\label{eqn:levi-civita}
\begin{aligned}
\overline\nabla_{E_1}E_1&=\tfrac{\kappa y}{2}E_2,&\overline\nabla_{E_1}E_2&=-\tfrac{\kappa y}{2}E_1+\tau E_3,&\overline\nabla_{E_1}E_3&=-\tau E_2,\\
\overline\nabla_{E_2}E_1&=-\tfrac{\kappa x}{2}E_2-\tau E_3,&\overline\nabla_{E_2}E_2&=\tfrac{\kappa x}{2}E_1,&\overline\nabla_{E_2}E_3&=\tau E_1,\\
\overline\nabla_{E_3}E_1&=-\tau E_2,&\overline\nabla_{E_3}E_2&=\tau E_1,&\overline\nabla_{E_3}E_3&=0.
\end{aligned}
\end{equation}

If $\kappa<0$, a different global model for $\E(\kappa,\tau)$ is the half-space model defined as $\{(x,y,z)\in\R^2:y>0\}$ endowed with the Riemannian metric
\begin{equation}\label{eqn:halfspace-metric}\frac{\df x^2+\df y^2}{-\kappa y^2}+\left(\df z+\frac{2\tau}{\kappa y}\df x\right)^2.
\end{equation}
The conformal factor $\frac{1}{y\sqrt{-\kappa}}$ defines a metric of constant curvature $\kappa$ in the upper half-plane, and the orientation is chosen such that the frame
\begin{equation}\label{eqn:halfspace-frame}
\begin{aligned}
  E_1&=y\sqrt{-\kappa}\,\partial_x+\tfrac{2\tau}{\sqrt{-\kappa}}\partial_z,&
  E_2&=y\sqrt{-\kappa}\,\partial_y,&
  E_3&=\partial_z,
\end{aligned}\end{equation}
is orthonormal and positively oriented. The Killing submersion still reads $\pi(x,y,z)=(x,y)$ in this model (with unit Killing vector field $\xi=E_3$). A global isometry from the Cartan model to the half-space model is given in~\cite[\S2.3.2]{CMT}.

If $\kappa>0$, a global model for $\E(\kappa,\tau)$ is given by the unit $3$-sphere $\mathbb{S}^3 = \{(z,w)\in \C^2:\, |z|^2 +|w|^2 = 1\}$ equipped with the Riemannian metric
\[
  \df s^2(X, Y) = \tfrac{4}{\kappa}\left[\langle{X},{Y}\rangle + \tfrac{16\tau^2}{\kappa^2}\bigl(\tfrac{4\tau^2}{\kappa} - 1\bigl)\langle{X},{\xi}\rangle\langle{Y},{\xi}\rangle \right],
\] 
being $\langle{\cdot},{\cdot}\rangle$ the usual inner product in $\mathbb{C}^2\equiv\mathbb{R}^4$. The unitary vector field $\xi$ is defined by $\xi_{(z, w)} = \frac{\kappa}{4\tau}(iz, iw)$ and the Killing submersion is the Hopf fibration 
\[\pi: \E(\kappa, \tau) \rightarrow \mathbb{S}^2(\kappa) \subset \mathbb{C} \times  \mathbb{R} \equiv\mathbb{R}^3,\qquad \pi(z, w) = \tfrac{2}{\sqrt{\kappa}}\bigl(z\bar{w}, \tfrac{1}{2}(|z|^2 - |w|^2) \bigr),\]
see~\cite[\S2]{Tor12}. Vertical fibers are compact and have length $\frac{8\pi\tau}{\kappa}$. This model is related to the Cartan model by the Riemannian covering map (in particular, it is a local isometry) $\Theta:M(\kappa,\tau)\to\mathbb{S}^3-\{(e^{i \theta},0)\colon  \theta \in \mathbb{R}\}$ given by
\begin{equation}\label{eq:local-isometry-Daniel-Berger}
\begin{aligned}
\Theta(x, y, z) &= \frac{1}{\sqrt{1 + \tfrac{\kappa}{4}(x^2 + y^2))}} \left(\tfrac{\sqrt{\kappa}}{2}(y + ix) \exp(i \tfrac{\kappa}{4\tau}z), \exp(i \tfrac{\kappa}{4\tau} z)\right).
\end{aligned}
\end{equation}

\section{Invariant horizontal $H$-tubes}\label{sec:invariant}

Let $\Gamma\subset\E(\kappa,\tau)$ be a fixed horizontal geodesic, and consider the continuous $1$-parameter group of translations $\{\Phi_t\}_{t\in\R}$ along $\Gamma$, see~\cite[Lem.~2.1]{CMT}. The rest of isometries that leave $\Gamma$ invariant are the axial symmetries about any vertical or horizontal geodesic orthogonal to $\Gamma$. Our goal in this section is to analyze the \textsc{ode} associated to $H$-surfaces invariant by $\{\Phi_t\}_{t\in\R}$ to prove the following lemma. These surfaces will be called $H$-tubes as explained in the introduction.

\begin{lemma}\label{lemma:existence}
Let $\kappa,\tau,H\in\R$ such that $\kappa+4H^2>0$. There is a unique $H$-surface $T_H$ immersed in $\mathbb{E}(\kappa,\tau)$ invariant by any isometry of $\E(\kappa,\tau)$ that leaves $\Gamma$ invariant (uniqueness is up to vertical translations of length $\frac{2\pi\tau}{\kappa}$ in Berger spheres). The surface $T_H$ is topologically a cylinder if $\kappa\leq 0$ or a torus if $\kappa>0$.
\end{lemma}

It is natural to assume that $4H^2+\kappa>0$ in this construction (i.e., we assume the mean curvature is supercritical), since we will show that there are no $H$-surfaces at bounded distance from $\Gamma$ if $4H^2+\kappa\leq 0$, see Corollary~\ref{coro:halfspace}. 

Recall that theses surfaces have been discussed with different parametrizations in~\cite{Pedrosa,Onnis,Man13,Pena,Vrzina,Kase}, but we are interested in a particular one that will help us understand whether they foliate or not. We will assume that $\tau\neq 0$ and distinguish cases depending on the sign of $\kappa$, since the case $\tau=0$ can be recovered as a limit. It is important to notice that there is no natural notion of height when $\tau\neq 0$. To overcome this issue, we will parametrize $\Gamma$ as $\gamma(s)$ and a orthogonal horizontal geodesic as $\alpha(s)$, both having unit-speed. The vertical surface $\pi^{-1}(\alpha)$ contains the subset of the orbit space (for the action of the isometry group) in which the profile curves of the $H$-tubes are defined. Moreover, $\pi^{-1}(\alpha)$ is parametrized by the horizontal distance $r$ along $\alpha$ and the vertical distance $h$ over $\alpha$. All these geometric functions adopt different expressions depending on the sign of $\kappa$.

\medskip
\noindent\textbf{Case A. Berger Spheres ($\kappa>0$).} Consider the unit-speed horizontal geodesics
  \begin{align*}
  \gamma(s)&=\left(\tfrac{2}{\sqrt{\kappa}}\cos(s\sqrt{\kappa}),\tfrac{2}{\sqrt{\kappa}}\sin(s\sqrt{\kappa}),\tfrac{2\tau}{\sqrt{\kappa}}s\right),\\
  \alpha(s)&=\left(\tfrac{2}{\sqrt{\kappa}}\tan(s\tfrac{\sqrt{\kappa}}{2}+\tfrac{\pi}{4}),0,0\right),
  \end{align*}
  in the Cartan model. They are orthogonal at $(\tfrac{2}{\sqrt{\kappa}},0,0)=\gamma(0)=\alpha(0)$. Translations along $\gamma$ are given by the $1$-parameter group of isometries
  \[\Phi_t(x,y,z)=\left(x\cos(t)-y\sin(t),x\sin(t)+y\cos(t),z+\tfrac{2\tau}{\kappa}t\right).\]
  A surface invariant by the action of $\Phi_t$ can be parametrized in terms of the coordinates $r$ (the horizontal distance along $\alpha$) and $h$ (the height over $\alpha$) as
  \begin{equation}\label{eqn:X:berger}
  X(u,v)=\left(\tfrac{2}{\sqrt{\kappa}}\tan(\tfrac{\sqrt{\kappa}}{2}r(u)+\tfrac\pi4)\cos(v),\tfrac{2}{\sqrt{\kappa}}\tan(\tfrac{\sqrt{\kappa}}{2}r(u)+\tfrac\pi4)\sin(v),h(u)+\tfrac{2\tau}{\kappa}v\right).
  \end{equation}
  The mean curvature of~\eqref{eqn:X:berger} is really cumbersome. However, there is a trick that simplifies the computations. The area element $W^2=\langle X_u,X_u\rangle\langle X_v,X_v\rangle-\langle X_u,X_v\rangle^2$ can be computed as
  \[W^2=\frac{\cos^2(r(u)\sqrt{\kappa})}{\kappa}h'(u)^2+\frac{4\tau^2+(\kappa-4\tau^2)\cos^2(r(u)\sqrt{\kappa})}{\kappa^2}r'(u)^2.\]
  We can employ a new auxiliary smooth function $\varphi$ such that
  \begin{equation}\label{eqn:rh:berger}
  r'(u)=\frac{-\sin(\varphi(u))}{\sqrt{4\tau^2+(\kappa-4\tau^2)\cos^2(r(u)\sqrt{\kappa})}},\qquad h'(u)=\frac{\cos(\varphi(u))}{\sqrt{\kappa}\cos(r(u)\sqrt{\kappa})}.
  \end{equation}   
  This is nothing but an arc-length reparametrization of $u\mapsto(r(u),h(u))$ for a certain Riemannian metric in $\R^2$ that makes the denominator of the mean curvature equation constant. After some long computations, we find that the surface parametrized by $X$ has constant mean curvature $H$ if and only if
  \begin{equation}\label{eqn:theta:berger}
  \varphi'(u)=\frac{2H+\sqrt{\kappa}\cos(\varphi(u))\tan(r(u)\sqrt{\kappa})}{\sqrt{4\tau^2+(\kappa-4\tau^2)\cos^2(r(u)\sqrt{\kappa})}}.
  \end{equation} 
  This gives an \textsc{ode} system on $h$, $r$ and $\varphi$, such that the energy function
  \[E=\cos(r(u)\sqrt{\kappa})\cos(\varphi(u))-\tfrac{2H}{\sqrt{\kappa}}\sin(r(u)\sqrt{\kappa})\]
  is constant along solutions (i.e., it does not depend on $u$). We would like the surface to be invariant by axial symmetry with respect to the $z$-axis (which is a vertical geodesic), since this isometry leaves $\Gamma$ invariant. This implies that there must be some $u$ such that $r(u)=0$ and $h'(u)=0$, in which case $\cos(\varphi(u))=0$ by~\eqref{eqn:rh:berger}. This gives a unique candidate value $E=0$ to solve Lemma~\ref{lemma:existence}, and we shall see in what follows that this value works. Note that $E=0$ leads to
  \begin{equation}\label{eqn:ru:berger}
  r(u)=\tfrac{1}{\sqrt{\kappa}}\arctan(\tfrac{\sqrt{\kappa}}{2H}\cos(\varphi(u))).
  \end{equation}
  Replacing $u$ by $\varphi$ as the parameter, it follows that
  \begin{align*}
  \frac{\mathrm{d}r}{\mathrm{d}\varphi}&=\frac{r'(u)}{\varphi'(u)}=\frac{-\sin(\varphi)}{2H+\sqrt{\kappa}\cos(\varphi)\tan(r\sqrt{\kappa})}=\frac{-2H\sin(\varphi)}{4H^2+\kappa\cos^2(\varphi)},\\
  \frac{\mathrm{d}h}{\mathrm{d}\varphi}&=\frac{h'(u)}{\varphi'(u)}=\frac{\frac{\cos(\varphi)}{\sqrt{2\kappa}\cos(r\sqrt{\kappa})}}{\frac{2H+\sqrt{\kappa}\cos(\varphi)\tan(r\sqrt{\kappa})}{\sqrt{4\tau^2+(\kappa-4\tau^2)\cos^2(r\sqrt{\kappa})}}}=\frac{2\cos(\varphi)\sqrt{H^2+\tau^2\cos^2(\varphi)}}{4H^2+\kappa\cos^2(\varphi)}
  \end{align*}
  These two equations can be integrated explicitly. The function $r$ can be actually recovered from~\eqref{eqn:ru:berger}, whereas $h(r)$ is determined up to a constant of integration. This constant is chosen such that such that $h(0)=0$ so that the surface is invariant by the axial symmetry about $\alpha$. If $H>0$, this integration gives
  \begin{equation}\label{eqn:parametrization:berger}
  \begin{aligned}
    r(\varphi)&=\tfrac{1}{\sqrt{\kappa}}\arctan(\tfrac{\sqrt{\kappa}}{2H}\cos(\varphi)),\\
    h(\varphi)&=\tfrac{2H\sqrt{\kappa-4\tau^2}}{\kappa\sqrt{4H^2+\kappa}}\arctanh\frac{H\sqrt{\kappa-4\tau^2}\sin(\varphi)}{\sqrt{4H^2+\kappa}\sqrt{H^2+\tau^2\cos^2(\varphi)}}\\
    &\qquad+\tfrac{4\tau}{\kappa}\arctan\frac{\tau\sin(\varphi)}{\sqrt{H^2+\tau^2}+\sqrt{H^2+\tau^2\cos^2(\varphi)}}
  \end{aligned}\end{equation}
  We have also used the supercritical condition $4H^2+\kappa>0$ in the integration process. Formula~\eqref{eqn:parametrization:berger} holds for $\kappa-4\tau^2\geq 0$ (elliptic Berger spheres) and also for $\kappa-4\tau^2<0$ (hyperbolic Berger spheres) via the identity $\arctanh(ix)=i\arctan(x)$ valid for all $x\in\R$. Some of these curves are represented in Figure~\ref{fig:foliation-berger}. Note that $r(\varphi)$ and $h(\varphi)$ are $2\pi$-periodic functions defined for all $\varphi\in\R$ so we obtain an $H$-torus for all $H>0$. Since $r(\varphi)$ is an even function and $h(\varphi)$ is odd, it follows that the surface is invariant by axial symmetries about any vertical or horizontal geodesic orthogonal to $\Gamma$. The case $H=0$ also fits in the whole family, as we shall see in Theorem~\ref{thm:foliation}, but it needs a different parametrization.

  \begin{remark}\label{rmk:clifford}
  If $H=0$, consider the unit-speed horizontal geodesic 
  \[\beta(s)=\left(\tfrac{2}{\sqrt{\kappa}}\tan(s\tfrac{\sqrt{\kappa}}{2}+\tfrac{\pi}{4}),0,\tfrac{\pi\tau}{\kappa}\right),\]
  which is a vertical translation of $\alpha$. The surface $T_0$ spanned by the group $\Phi_t$ acting on $\beta$ is a (minimal) spherical helicoid of axes the vertical geodesics through $\beta(\frac{-\pi}{2\sqrt{\kappa}})$ and $\beta(\frac{\pi}{2\sqrt{\kappa}})$, which project to antipodal points of $\mathbb{M}^2(\kappa)$ by the Hopf fibration. Except at these antipodal points, each $p\in\M^2(\kappa)$ has four preimages by the projection $\pi|_{T_0}:T_0\to\M^2(\kappa)$. Note that the vertical geodesics of $T_0$ are the orbits of $\{\Phi_t\}_{t\in\R}$ which are not transversal to our orbit space $\pi^{-1}(\alpha)$, whence the orbits in $\E(\kappa,\tau)-T_0$ correspond to values $(r,h)$ such that $|r|<\frac{\pi}{2\sqrt{\kappa}}$ and $v$ not an odd multiple of $\frac{\tau\pi}{\kappa}$. The surface $T_0$ is a minimal torus (indeed, $T_0$ is a Clifford torus for the round metric). A more detailed description of the geometry of $T_0$ can be found in Section~\ref{sec:spherical-helicoids}, see also~\cite[\S3]{Tor12} or~\cite[Ex.~3.16]{CMT}.
  \end{remark}

  Formula~\eqref{eqn:parametrization:berger} also holds true in the limit case $\tau=0$ (provided that $H>0$) and gives a parametrization of the $H$-tori in $\mathbb{S}^2\times\R$ different from that in~\cite{Man13}. 
  
  \begin{figure}
    \centering\includegraphics[width=0.8\textwidth]{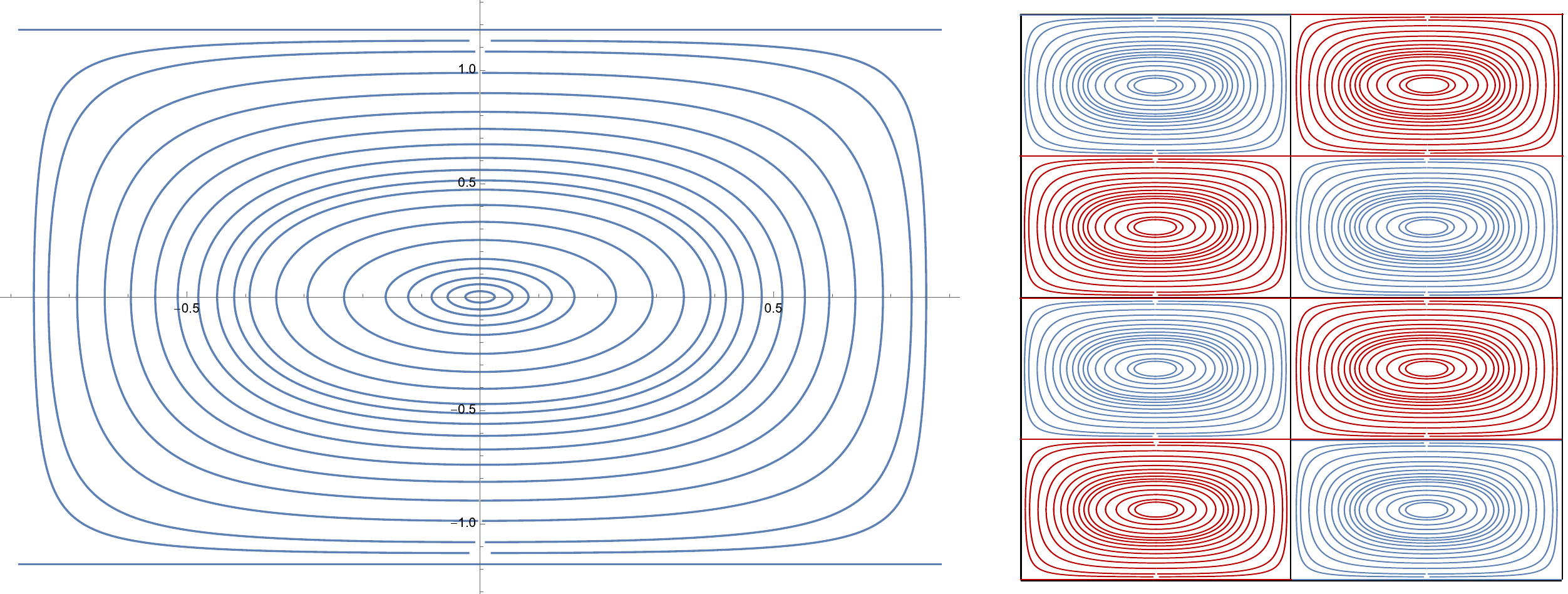}
    \centering\includegraphics[width=0.8\textwidth]{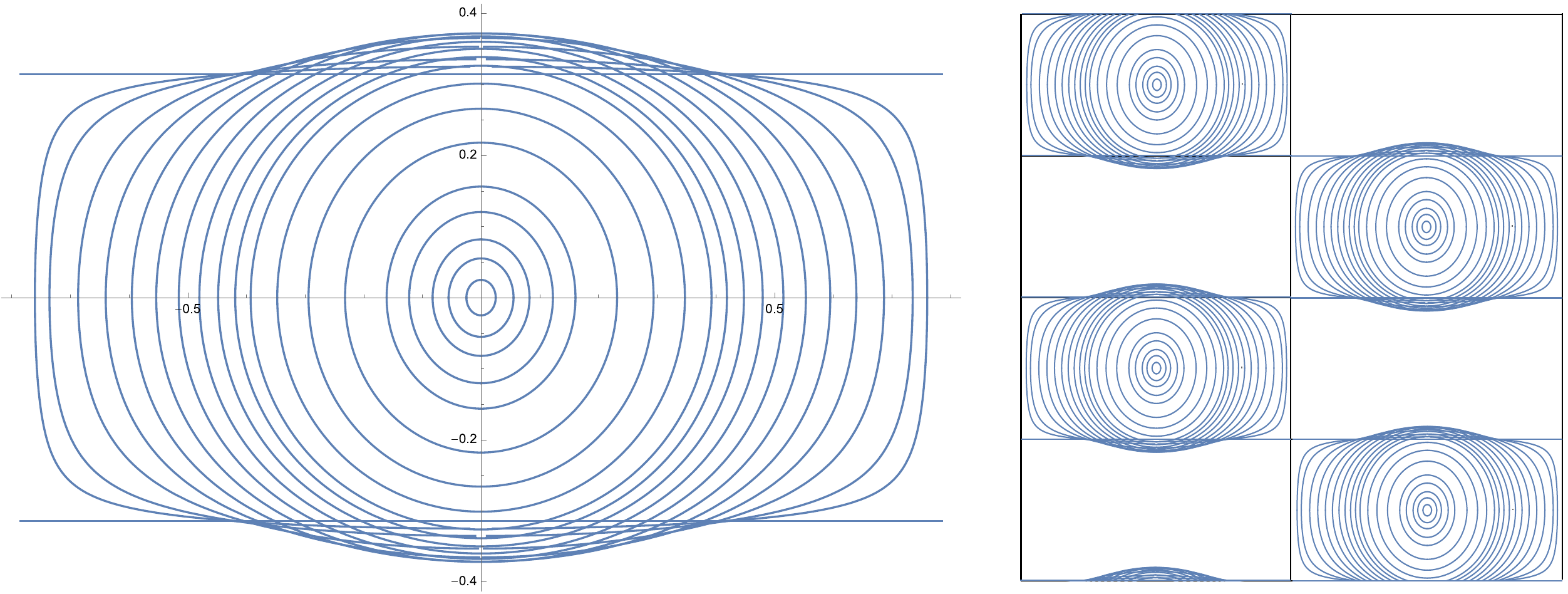}
    \centering\includegraphics[width=0.8\textwidth]{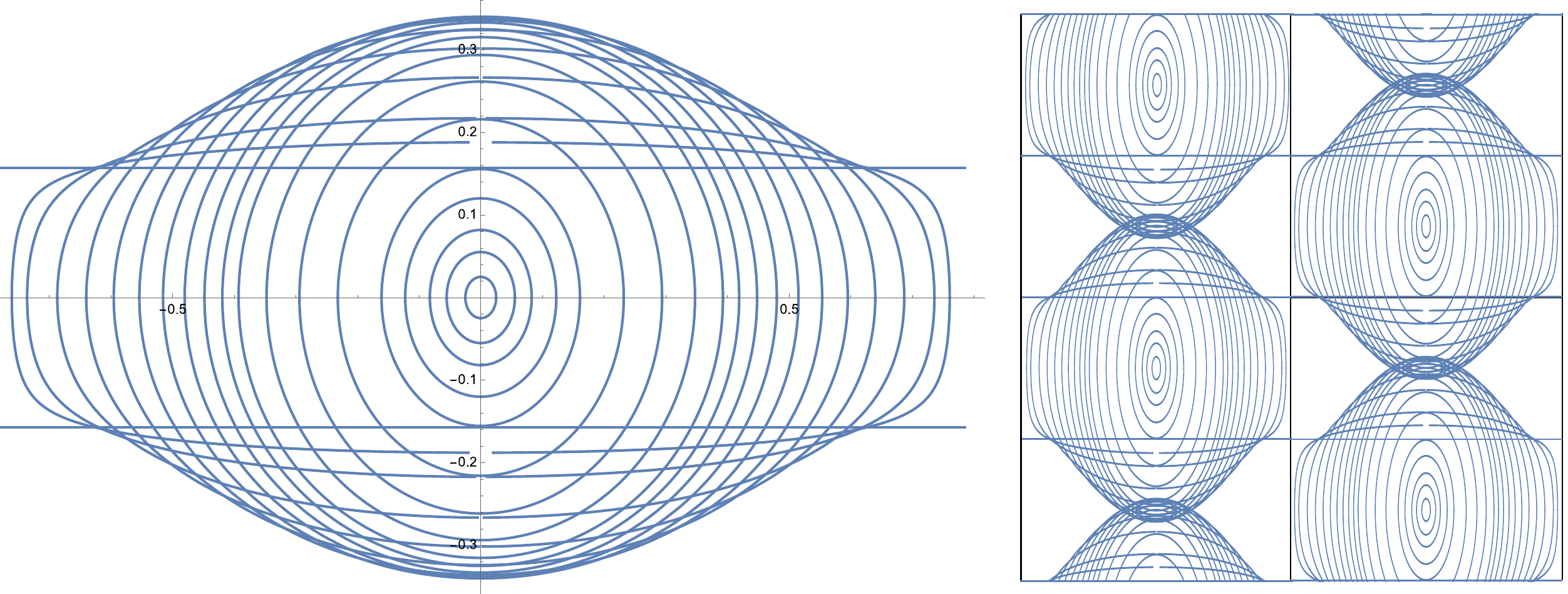}
    \caption{Profile curves of $H$-tori in the hyperbolic Berger sphere $\E(4,1.5)$ (top) and the elliptic Berger spheres $\E(4,0.4)$ (center) and $\E(4,0.2)$ (bottom). On the right, we can see the rescaled intersection of $\mathcal T$ and the Clifford torus $\pi^{-1}(\alpha)$. In the hyperbolic case (top), we have also drawn in red color the translated curves (corresponding to $H<0$) that complete the family $\mathcal T$. In the central image, we can see that the foliation fails, whilst in bottom picture we can even find non-embedded $H$-tubes.
    }\label{fig:foliation-berger}
  \end{figure}

\medskip
\noindent\textbf{Case B. Heisenberg group ($\kappa=0$).} In this case the computations are considerably more simple. The surfaces we will obtain here have already been parametrized as bigraphs in~\cite[Thm.~6]{FMP}. We will follow the same argument as in case A and consider unit-speed horizontal geodesics
 \[\gamma(s)=(s,0,0),\qquad \alpha(s)=(0,s,0),\] 
which are orthogonal at their intersection point $(0,0,0)=\gamma(0)=\alpha(0)$. Translations along $\gamma$ are given by the group of isometries $\Phi_t(x,y,z)=(x+t,y,z+\tau y t)$, whence a surface invariant by $\Phi_t$ can be parametrized in terms of the horizontal distance $r$ along $\alpha$ and the vertical distance $h$ over $\alpha$ as:
\begin{equation}\label{eqn:X:nil}
X(u,v)=(v,r(u),h(u)+\tau v r(u)).
\end{equation}
We can introduce an auxiliary function $\varphi$ such that the constant mean curvature $H$ equation for this parametrization reads
 \[r'(u)=\frac{-\sin(\varphi(u))}{\sqrt{1+4\tau^2r(u)^2}},\qquad h'(u)=\cos(\varphi(u)),\qquad \varphi'(u)=\frac{2H}{\sqrt{1+4\tau^2r(u)^2}}.\]
The energy $E=\cos(\varphi(u))-2Hr(u)$ is constant along solutions, and the desired $H$-tubes must satisfy $E=0$ by the same reasons as in case A. Since $E=0$ implies that $r(u)=\frac{1}{2H}\cos(\varphi(u))$, we can write \textsc{ode}s for $r$ and $h$ as functions of $\varphi$:
\begin{align*}
  \frac{\mathrm{d}r}{\mathrm{d}\varphi}&=\frac{r'(u)}{\varphi'(u)}=\tfrac{-1}{2H}\sin(\varphi),\\
  \frac{\mathrm{d}h}{\mathrm{d}\varphi}&=\frac{h'(u)}{\varphi'(u)}=\frac{\cos(\varphi)}{\frac{2H}{\sqrt{1+4\tau^2r^2}}}=\frac{\cos(\varphi)\sqrt{H^2+\tau^2\cos^2(\varphi)}}{2H^2}.
  \end{align*}
  This finally leads to the following explicit expressions:
  \begin{equation}\label{eqn:parametrization:Nil}
  \begin{aligned}
    r(\varphi)&=\tfrac{1}{2H}\cos(\varphi),\\
    h(\varphi)&=\tfrac{H^2+\tau ^2}{4 H^2 \tau }\arcsin\frac{\tau\sin(\varphi)}{\sqrt{H^2+\tau ^2}}+\tfrac{1}{4 H^2}  \sin(\varphi) \sqrt{H^2+\tau^2\cos
     ^2(\varphi)}
    \end{aligned}\end{equation}
  Note that the supercritical condition $\kappa+4H^2>0$ is relevant since it excludes the case $H=0$ in which the above solution is not defined. Since the parametrization~\eqref{eqn:parametrization:Nil} has the desired symmetries, it follows that it is the (unique) surface that solves Lemma~\ref{lemma:existence}.

\begin{figure}
    \centering\includegraphics[height=8cm]{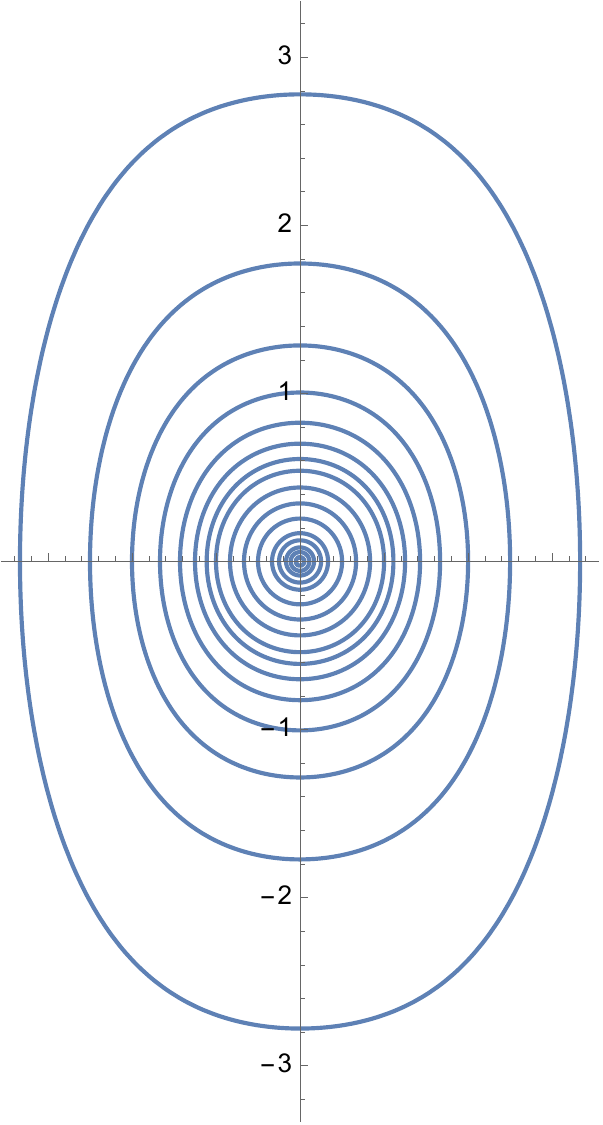}\quad\includegraphics[height=8cm]{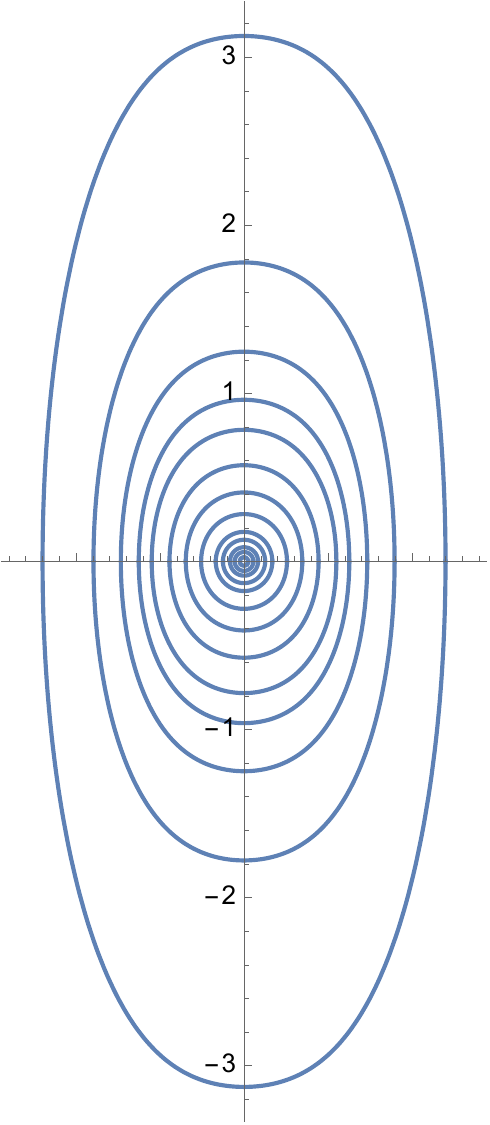}
    \caption{Profiles of $H$-cylinders in Heisenberg space $\E(0,0.5)$ (left) and in the special linear group $\E(-1,1)$ (right).}
  \end{figure}

\medskip
\noindent\textbf{Case C. Special linear group ($\kappa<0$).} In this last case, the $H$-cylinders we are looking for have been previously spotted by Vr\v{z}ina~\cite{Vrzina} (using a nonexplicit continuity method) and also belong to the families of $H$-surfaces invariant by hyperbolic translations discovered by Peñafiel~\cite{Pena}. We will employ the halfspace model in which hyperbolic translations have the simple expression $\Phi_t(x,y,z)=(e^tx,e^ty,z)$, which contains the horizontal geodesic orbit $\gamma(t)=(0,e^{t\sqrt{-\kappa}},0)$ parametrized by unit speed. The horizontal unit-speed geodesic $\alpha(t)$ intersecting $\gamma(t)$ orthogonally at $(0,1,0)=\gamma(0)=\alpha(0)$ is given by
\begin{align*}
  \alpha(s)&=\left(\tanh(s\sqrt{-\kappa}),\sech(s\sqrt{-\kappa}),\tfrac{4\tau}{\kappa}\arccos\frac{\tanh(\tfrac{\sqrt{-\kappa}}{2}s)}{\sqrt{1+\tanh^2(\tfrac{\sqrt{-\kappa}}{2}s)}}\right).
\end{align*}
A surface invariant by $\Phi_t$ can be parametrized in terms of the functions $r$ and $h$ as
\begin{equation}\label{eqn:X:SL}
\begin{aligned}
X(u,v)=\Biggl(e^v\tanh(r(u)\sqrt{-\kappa}),\ &e^v\sech(r(u)\sqrt{-\kappa}),\\
&h(u)+\tfrac{4\tau}{\kappa}\arccos\frac{\tanh(r(u)\frac{\sqrt{-\kappa}}{2})}{\sqrt{1+\tanh^2(r(u)\frac{\sqrt{-\kappa}}{2})}}\Biggr)\end{aligned}
\end{equation}
As in the previous cases, we introduce an auxiliary function $\varphi$, so that the condition of constant mean curvature $H\in\R$ translates into the \textsc{ode} system
\begin{align*}
r'(u)&=\frac{-\sin(\varphi(u))}{\sqrt{-4\tau^2-(\kappa-4\tau^2)\cosh{\!}^2(r(u)\sqrt{-\kappa}))}}\\
h'(u)&=\frac{\cos(\varphi(u))}{\sqrt{-\kappa}\cosh(r(u)\sqrt{-\kappa})}\\
\varphi'(u)&=\frac{2H-\sqrt{-\kappa}\cos(\varphi(u))\tanh(r(u)\sqrt{-\kappa})}{\sqrt{-4\tau^2-(\kappa-4\tau^2)\cosh{\!}^2(r(u)\sqrt{-\kappa}))}}
\end{align*}
The following energy is a first integral:
\[E=\cosh(r(u)\sqrt{-\kappa})\cos(\varphi(u))-\tfrac{2H}{\sqrt{-\kappa}}\sinh(r(u)\sqrt{-\kappa})\]
and $H$-tubes correspond to $E=0$, which gives the substitution
\[r(u)=\tfrac{1}{\sqrt{-\kappa}}\arctanh\left(\tfrac{\sqrt{-\kappa}}{2H}\cos(\varphi(u))\right).\]
Using the parameter $\varphi$ instead of $u$, we get
\begin{align*}
  \frac{\mathrm{d}r}{\mathrm{d}\varphi}&=\frac{r'(u)}{\varphi'(u)}=\frac{\frac{-\sin(\varphi)}{\sqrt{-4\tau^2-(\kappa-4\tau^2)\cosh{\!}^2(r(u)\sqrt{-\kappa}))}}}{\frac{2H-\sqrt{-\kappa}\cos(\varphi)\tanh(r\sqrt{-\kappa})}{\sqrt{-4\tau^2-(\kappa-4\tau^2)\cosh{\!}^2(r(u)\sqrt{-\kappa}))}}}=\frac{-2H\sin(\varphi)}{4H^2+\kappa\cos^2(\varphi)},\\
  \frac{\mathrm{d}h}{\mathrm{d}\varphi}&=\frac{h'(u)}{\varphi'(u)}=\frac{\frac{\cos(\varphi)}{\sqrt{-\kappa}\cosh(r\sqrt{-\kappa})}}{\frac{2H-\sqrt{-\kappa}\cos(\varphi)\tanh(r\sqrt{-\kappa})}{\sqrt{-4\tau^2-(\kappa-4\tau^2)\cosh{\!}^2(r(u)\sqrt{-\kappa}))}}}=\frac{2\cos(\varphi)\sqrt{H^2+\tau^2\cos^2(\varphi)}}{4H^2+\kappa\cos^2(\varphi)}.
  \end{align*}
  This finally leads to the following explicit expressions:
  \begin{equation}\label{eqn:parametrization:SL}
  \begin{aligned}
    r(\varphi)&=\tfrac{1}{\sqrt{-\kappa}}\arctanh(\tfrac{\sqrt{-\kappa}}{2H}\cos(\varphi)),\\
    h(\varphi)&=\tfrac{2H\sqrt{-\kappa+4\tau^2}}{\kappa\sqrt{4H^2+\kappa}}\arctan\left(\frac{H\sqrt{-\kappa+4\tau^2}\sin(\varphi)}{\sqrt{4H^2+\kappa}\sqrt{H^2+\tau^2\cos^2(\varphi)}}\right)\\
    &\qquad-\tfrac{4\tau}{\kappa}\arctan\left(\frac{\tau\sin(\varphi)}{\sqrt{H^2+\tau^2}+\sqrt{H^2+\tau^2\cos^2(\varphi)}}\right)
    \end{aligned}\end{equation}
Observe that $-\kappa+4\tau^2>0$ since $\kappa<0$; also, the supercritical condition $4H^2+\kappa>0$ plays an important role in~\eqref{eqn:parametrization:SL}. Since $h(\varphi)$ and $r(\varphi)$ are $2\pi$-periodic and have the desired symmetries again, they prove Lemma~\ref{lemma:existence} in the case $\kappa<0$.

\section{Embeddedness and foliation}

Throughout this section, we will fix $\kappa,\tau\in\R$. We will study the profile curves $\alpha_H(\varphi)=(r_H(\varphi),h_H(\varphi))$ defined by~\eqref{eqn:parametrization:berger},~\eqref{eqn:parametrization:Nil} and~\eqref{eqn:parametrization:SL}, where we will employ an additional subindex $H$ to indicate the dependence on the mean curvature. Note that all cases can be treated together (regardless the sign of $\kappa$) because we have seen that the derivatives of $r_H$ and $h_H$ admit unified expressions.

Note that a normal to $\alpha_H$ in the flat $(r,h)$-plane is given by
\[\eta_H(\varphi)=(-h'_H(\varphi),r'_H(\varphi))=\left(\frac{-2\cos(\varphi)\sqrt{H^2+\tau^2\cos^2(\varphi)}}{4H^2+\kappa\cos^2(\varphi)},\frac{-2H\sin(\varphi)}{4H^2+\kappa\cos^2(\varphi)}\right),\]
whose components have the same signs as $(-\cos(\varphi),-\sin(\varphi))$. Since $\eta_H(\varphi)$ never vanishes, we deduce that $\eta_H$ has winding number $1$ around the origin. Moreover,
\[h_H'(\varphi)r_H''(\varphi)-h_H''(\varphi)r_H'(\varphi)=\frac{H(H^4+H^2\tau^2x^2(3-x^2)+\tau^2x^4(2-x^2))}{4(H^2+\tau^2x^2)^{3/2}(H^2+x^2)^2}\geq 0,\]
where $x=\cos(\varphi)$, so $\alpha_H$ is a convex curve in the $(r,h)$-plane, and hence embedded.

However, the embeddedness of the tube $T_H$ might fail after translating $\alpha_H$ along the geodesic $\Gamma$ when the fibers have finite length, i.e, if the ambient space is a Berger sphere. The fibers of a Berger sphere have length $\frac{8\pi\tau}{\kappa}$ and the third coordinate of $X(u,v)$ in~\eqref{eqn:X:berger} increases by $\frac{4\pi\tau}{\kappa}$ whenever we add $2\pi$ to the parameter $v$, whence the embeddedness of $T_H$ fails if and only if the maximum of $h_H(\varphi)$ is greater than or equal to $\frac{2\pi\tau}{\kappa}$ (this condition also prevents the curve $\alpha_H$ from intersecting itself before translations, see Figure~\ref{fig:foliation-berger}).

\begin{theorem}\label{thm:foliation}
Let $\Gamma\subset\E(\kappa,\tau)$ be a horizontal. The family $\mathcal T=\{T_H:4H^2+\kappa>0\}$ of $H$-tubes around $\Gamma$ produces a foliation if and only if $(1-x_0^2)\kappa-4\tau^2\leq 0$, where $x_0\approx 0.833557$ is the unique positive solution of the equation $x\arctanh(x)=1$.
\begin{enumerate}[label=\emph{(\alph*)}]
  \item If $\kappa\leq 0$, then it foliates $\E(\kappa,\tau)-\Gamma$.
  \item If $\kappa\geq 0$, then it foliates $\E(\kappa,\tau)-(\Gamma\cup\Gamma')$, where $\Gamma'$ is the horizontal geodesic that differs from $\Gamma$ in a vertical translation of length $\frac{2\tau\pi}{\kappa}$. Note that in this case $H$ varies from $-\infty$ to $+\infty$.
\end{enumerate}
\end{theorem}

\begin{proof}
Consider the family $\mathcal T$ with the additional restriction $H>0$ (the case $H=0$ will be added \emph{ad hoc} later). As $H\to\infty$, both $r_H(\varphi)$ and $h_H(\varphi)$ converge uniformly to zero whereas $\alpha_H(\varphi)$ stays away from $(0,0)$ for all $H$ such that $4H^2+\kappa>0$. Since the family of curves $\alpha_H$ is continuous with respect to $H$, it will fail to foliate an open subset of the $(r,h)$-plane if and only if we can find $0<H_1<H_2$ such that $\alpha_{H_1}$ is tangent to $\alpha_{H_2}$ at some point, so there are $\varphi_1,\varphi_2\in\R$ and $\rho>0$ such that 
\begin{equation}\label{eqn:foliation:eqn1}
\begin{aligned}
(r_{H_2}(\varphi_2),h_{H_2}(\varphi_2))&=(r_{H_1}(\varphi_1),h_{H_1}(\varphi_1)),\\
 (r_{H_2}'(\varphi_2),h'_{H_2}(\varphi_2))&=(\rho\, r_{H_1}'(\varphi_1),\rho\, h'_{H_1}(\varphi_1)).
\end{aligned}
\end{equation}
We can further assume that $\varphi_1,\varphi_2\in[0,\frac\pi2]$ by the symmetries of the curves.

The first component of the first equation of~\eqref{eqn:foliation:eqn1} gives $\frac{\cos(\varphi_1)}{H_1}=\frac{\cos(\varphi_2)}{H_2}$, so the second component of the second equation of~\eqref{eqn:foliation:eqn1} can be written as
\begin{align*}
\cos(\varphi_1)\sqrt{H_1^2+\tau^2\cos^2(\varphi_1)}&=\rho H_2\cos(\varphi_2)\sqrt{1+\tfrac{\tau^2}{H_2^2}\cos^2(\varphi_2)}\\
&=\rho H_2\cos(\varphi_2)\sqrt{1+\tau^2\tfrac{\tau^2}{H_1^2}\cos^2(\varphi_1)}\\
&=\rho\tfrac{H_2^2}{H_1^2}\cos(\varphi_1)\sqrt{H_1^2+\tau^2\cos^2(\varphi_1)}.
\end{align*}
It follows that $\cos(\varphi_1)=0$ if and only if $\cos(\varphi_2)=0$.
\begin{itemize}
  \item If $\cos(\varphi_1)\neq0$, then we deduce that $H_1^2=\rho H_2^2$. The first component of the second equation of~\eqref{eqn:foliation:eqn1} implies that
  \[H_1\sin(\varphi_1)=\rho H_2\sin(\varphi_2)=\tfrac{H_1^2}{H_2}\sin(\varphi_2),\]
  so we find that $\frac{\sin(\varphi_1)}{H_1}=\frac{\sin(\varphi_2)}{H_2}$. This finally gives
  \[\frac{1}{H_1^2}=\frac{\cos^2(\varphi_1)}{H_1^2}+\frac{\sin^2(\varphi_1)}{H_1^2}=\frac{\cos^2(\varphi_2)}{H_2^2}+\frac{\sin^2(\varphi_2)}{H_2^2}=\frac{1}{H_2^2},\]
  whence $H_1=H_2$ and we are done (this means no such tangency point can occur at points other than $\varphi_1=\varphi_2=\frac{\pi}{2}$ regardless the values of $\kappa$ and $\tau$).

  \item If $\cos(\varphi_1)=0$ and $\cos(\varphi_2)=0$, then $\varphi_1=\varphi_2=\frac{\pi}{2}$ and $r_{H_1}(\varphi_1)=r_{H_2}(\varphi_2)=0$. Therefore, the family $\mathcal T$ produces a foliation if and only if $H\mapsto h_H(\frac{\pi}{2})$ is one-to-one. We will distinguish two cases:
  \begin{enumerate}
    \item If $\kappa-4\tau^2>0$, it follows from~\eqref{eqn:parametrization:berger} that
  \begin{equation}\label{eqn:foliation:eqn2}
  \frac{\partial h_H(\frac{\pi}{2})}{\partial H}=\frac{2}{4H^2+\kappa}\left(\frac{\sqrt{\kappa-4\tau^2}}{\sqrt{4H^2+\kappa}}\arctanh\left(\frac{\sqrt{\kappa-4\tau^2}}{\sqrt{4H^2+\kappa}}\right)-1\right).
  \end{equation}
  Since $x\arctanh(x)=1$ has a unique solution $x_0\approx 0.833557$, we infer from~\eqref{eqn:foliation:eqn2} that $H\mapsto h_H(\frac{\pi}{2})$ is increasing for $H<H_0$ and decreasing for $H>H_0$. This assumes that $\frac{\kappa-4\tau^2}{4H_0^2+\kappa}=x_0^2$, or equivalently $(1-x_0^2)\kappa-4\tau^2=4H_0^2x_0^2$, has a solution $H_0>0$. Such an $H_0$ exists if and only if $(1-x_0^2)\kappa-4\tau^2> 0$, in which case there is no foliation.
  
  \item If $\kappa-4\tau^2<0$, then we compute using~\eqref{eqn:parametrization:berger},~\eqref{eqn:parametrization:Nil} and~\eqref{eqn:parametrization:SL}:
  \begin{equation}\label{eqn:foliation:eqn3}
  \qquad\ \ \frac{\partial h_H(\frac{\pi}{2})}{\partial H}=\frac{-2}{4H^2+\kappa}\left(\frac{\sqrt{-\kappa+4\tau^2}}{\sqrt{4H^2+\kappa}}\arctan\left(\frac{\sqrt{-\kappa+4\tau^2}}{\sqrt{4H^2+\kappa}}\right)+1\right).
  \end{equation} 
  The right-hand side of~\eqref{eqn:foliation:eqn3} is negative, so we can say that the maximum height $h_H(\frac{\pi}{2})$ is strictly decreasing and $\mathcal T$ produces a foliation.
\end{enumerate}
\end{itemize}
It remains to show that the foliation actually covers $\E(\kappa,\tau)-\Gamma$. If $\kappa\leq 0$, this is easy since each $\alpha_H$ is convex and Equations~\eqref{eqn:parametrization:Nil} and~\eqref{eqn:parametrization:SL} yield the limits
\[\lim_{H\to 0}r_H(0)=\lim_{H\to 0}h_H(\tfrac{\pi}{2})=+\infty,\qquad \lim_{H\to +\infty}r_H(0)=\lim_{H\to +\infty}h_H(\tfrac{\pi}{2})=0.\]
Assume now that $\kappa>0$ and $(1-x_0^2)\kappa-4\tau^2<0$. Using~\eqref{eqn:parametrization:berger}, we find the following limits for a fixed value of $\varphi\in[0,\frac{\pi}{2}]$:
\begin{align*}
\lim_{H\to 0}h_H(\varphi)&=\tfrac{2\tau\varphi}{\kappa},& \lim_{H\to 0}r_H(\varphi)&=\begin{cases}\tfrac{\pi}{2\sqrt{\kappa}}&\text{if }\varphi\neq\frac{\pi}{2},\\0&\text{if }\varphi=\frac{\pi}{2},\end{cases}\\
\lim_{H\to+\infty}h_H(\varphi)&=0&\lim_{H\to+\infty}r_H(\varphi)&=0.
\end{align*}
By convexity and symmetry, we deduce that $\alpha_H$ foliates the whole open rectangle $(\frac{-\pi}{2\sqrt{\kappa}},\frac{\pi}{2\sqrt{\kappa}})\times(\frac{-\pi\tau}{\kappa},\frac{\pi\tau}{\kappa})$ minus the origin $(0,0)$, which represents the geodesic orbit $\Gamma$. By Remark~\ref{rmk:clifford}, the foliation is complete if we add the minimal torus $T_0$ as well as the vertical translation of length $\frac{2\tau\pi}{\kappa}$ of each $T_H$. These translated surfaces can be thought of as $T_{H}$ with $H<0$; in the $(r,h)$-plane, they foliate the rectangle $(\frac{-\pi}{2\sqrt{\kappa}},\frac{\pi}{2\sqrt{\kappa}})\times(\frac{\pi\tau}{\kappa},\frac{3\pi\tau}{\kappa})$ minus the point $(0,\frac{2\pi\tau}{\kappa})$, which corresponds to $\Gamma'$.
\end{proof}

\begin{corollary}\label{coro:halfspace}
If $4H^2+\kappa\leq 0$, there are no $H$-surfaces properly immersed in $\E(\kappa,\tau)$ at bounded distance from a horizontal geodesic $\Gamma$.
\end{corollary}

\begin{proof}
If $\kappa=0$, then the statement follows from the halfspace theorems of Hoffman--Meeks in Euclidean space~\cite{HM90} ($\tau=0$) and Daniel--Hauswirth in Heisenberg group~\cite{DH09} ($\tau\neq 0$), so we will assume that $\kappa<0$. The existence of the foliation given by Theorem~\eqref{thm:foliation} allows us to extend literally the proof of~\cite[Thm.~1.3]{MT22} using Mazet's general halfspace theorem~\cite[Thm.~7]{Mazet2013}. Briefly, to apply this result we need that: (a) each $H$-tube is parabolic (this follows because $T_H$ has intrinsic linear area growth), (b) the second fundamental form of each $C_H$ is uniformly bounded (this follows from the fact that the generating curve of $T_H$ is compact), and (c) each pair of $H$-tubes $T_H$ and $T_{H'}$ are uniformly quasi-isometric (such a quasi-isometry consists in sending points of $T_H$ to points of $T_{H'}$ with the same coordinates $(\varphi,v)$).
\end{proof}

\section{Sister minimal tori in Berger spheres}

Let $\phi:\Sigma\to\E(\kappa,\tau)$ be an isometric immersion of a two-sided Riemannian surface $\Sigma$. A global smooth unit normal $N$ defines the shape operator $Av = -\overline\nabla_vN$ for all tangent vectors $v$, as well as the angle function $\nu=\langle N,\xi\rangle\in\mathcal C^\infty(\Sigma)$. We can also consider the tangent part of the Killing field $T=\xi-\nu N\in\X(\Sigma)$. The orientation in $\E(\kappa,\tau)$ and the choice of $N$ induce an orientation in $\Sigma$ expressed in terms of a $\frac{\pi}{2}$-rotation $J$ in the tangent bundle such that $\{u,Ju,N_p\}$ is positively oriented (or equivalently $Ju=N_p\times u$) for all nonzero $u\in T_p\Sigma$. The immersion $\phi$ is univocally determined by the quadruplet $(A,T,J,\nu)$, called the fundamental data of the immersion, up to orientation-preserving isometries that also preserve the orientation of the fibers, see~\cite[Prop.~4.1]{Dan}.

Let $\kappa,\tau,H,\widetilde\kappa,\widetilde\tau,\widetilde H\in\mathbb{R}$ be such that $\kappa-4\tau^2=\widetilde\kappa-4\widetilde\tau^2$ and $\tau+iH=e^{i\theta}(\widetilde\tau+i\widetilde H)$ for some $\theta\in\R$. Given a simply-connected Riemannian surface $\Sigma$, Daniel~\cite{Dan} found an isometric correspondence between $\widetilde H$-immersions $\widetilde\phi:\Sigma\to\mathbb{E}(\widetilde\kappa,\widetilde\tau)$ and $H$-immersions $\phi:\Sigma\to\mathbb{E}(\kappa,\tau)$ by changing the fundamental data as
\begin{equation}\label{eqn:fundamental-rotation}
(A, T, J,\nu)=\bigl(\Rot_\theta\circ(\widetilde A-\widetilde H\,\id)+H\,\id,\Rot_\theta(\widetilde T),\widetilde J,\widetilde\nu\bigr),
\end{equation}
where $\Rot_\theta=\cos(\theta)\id+\sin(\theta)J$ is a rotation of angle $\theta$ in the tangent bundle of $\Sigma$. The immersions $\widetilde\phi$ and $\phi$ are called \emph{sister immersions}. In the case $\theta=\frac{\pi}{2}$, they are also called \emph{conjugate immersions} and represent a great tool to obtain $H$-surfaces in product spaces $\mathbb{M}^2(\kappa)\times\R$ starting with a minimal surface in a $\E(4H^2+\kappa,H)$, see~\cite{CMT} and the references therein. Note that one can restrict to the interval $\theta\in[0,\pi)$ since adding $\pi$ to the phase angle results in a trivial transformation that changes the sign of $\tau$, see~\cite[Rmk.~4]{CMT}.

\subsection{Behavior of vertical and horizontal geodesics}

In the sequel, we will consider two sister immersions $\widetilde\Sigma\looparrowright\E(\widetilde\kappa,\widetilde\tau)$ and $\Sigma\looparrowright\E(\kappa,\tau)$ related by a phase angle $\theta\in[0,\pi)$. For simplicity, we will denote by $\widetilde\Sigma$ and $\Sigma$ these immersed surfaces. We are interested in the relation between conjugate curves $\widetilde\gamma:[0,\ell]\to\widetilde\Sigma$ and $\gamma:[0,\ell]\to\Sigma$ by assuming that $\widetilde\Sigma$ is minimal $\widetilde\gamma$ is a vertical or horizontal geodesic. The case $\theta=\frac{\pi}{2}$ was analyzed first in~\cite{MT14,Ple14} but the following two lemmas are the first results in the general case and have independent proofs.

\begin{lemma}[Deformation of vertical geodesics]\label{lem:vertical-geodesics}
Let $\widetilde\gamma:[0,\ell]\to\widetilde\Sigma$ parametrize a vertical segment such that $\widetilde\gamma'=\widetilde\xi$. We can write the normal of $\widetilde\Sigma$ along $\widetilde\gamma$ as
\begin{equation}\label{lem:vertical-geodesics:eqn1}
\widetilde N_{\widetilde{\gamma}(t)}=\cos(\vartheta(t)) \widetilde E_1(t)+\sin(\vartheta(t))\widetilde E_2(t),
\end{equation}
for some $\vartheta\in\mathcal C^\infty([0,\ell])$, where $\{\widetilde E_1,\widetilde E_2,\widetilde E_3=\widetilde\xi\}$ is the positively oriented orthonormal frame in $\E(\widetilde\kappa,\widetilde\tau)$ given by~\eqref{eqn:standard-frame}.
\begin{enumerate}[label=\emph{(\alph*)}]
    \item The conjugate curve $\gamma$ satisfies $\langle\gamma',\xi\rangle=\cos(\theta)$.
    \item The projection $\alpha=\pi\circ\gamma$ has geodesic curvature $\kappa_g=2H-\frac{\vartheta'}{\sin(\theta)}$ as a curve of $\mathbb{M}^2(\kappa)$ with respect to $\pi_*N$ as unit normal along $\alpha$.

    \item Assume that $\widetilde R\subseteq\widetilde\Sigma$ and $R\subseteq\Sigma$ are conjugate regions on which $\nu>0$, and such that $\widetilde\gamma\subset\partial\widetilde R$ and $\gamma\subset\partial R$.
    \begin{itemize}
        \item If $\vartheta'>0$, then $J\widetilde\gamma'$ (resp.\ $J\gamma'$) is a unit outer conormal to $\widetilde R$ (resp.\ $R$) along $\widetilde\gamma$ (resp.\ $\gamma$), $\pi_*N$ points to the interior of $\pi(R)$ along $\alpha$, and $R$ locally approaches $\gamma$ from below (see Figure~\ref{fig:orientation}, top). 
        
        \item If $\vartheta'<0$, then $J\widetilde\gamma'$ (resp.\ $J\gamma'$) is a unit inner conormal to $\widetilde R$ (resp.\ $R$) along $\widetilde\gamma$ (resp.\ $\gamma$), $\pi_*N$ points to the exterior of $\pi(R)$ along $\alpha$, and $R$ locally approaches $\gamma$ from above (see Figure~\ref{fig:orientation}, bottom).
    \end{itemize}
\end{enumerate}
\end{lemma}

\begin{figure}
\begin{center}
\includegraphics[width=0.75\textwidth]{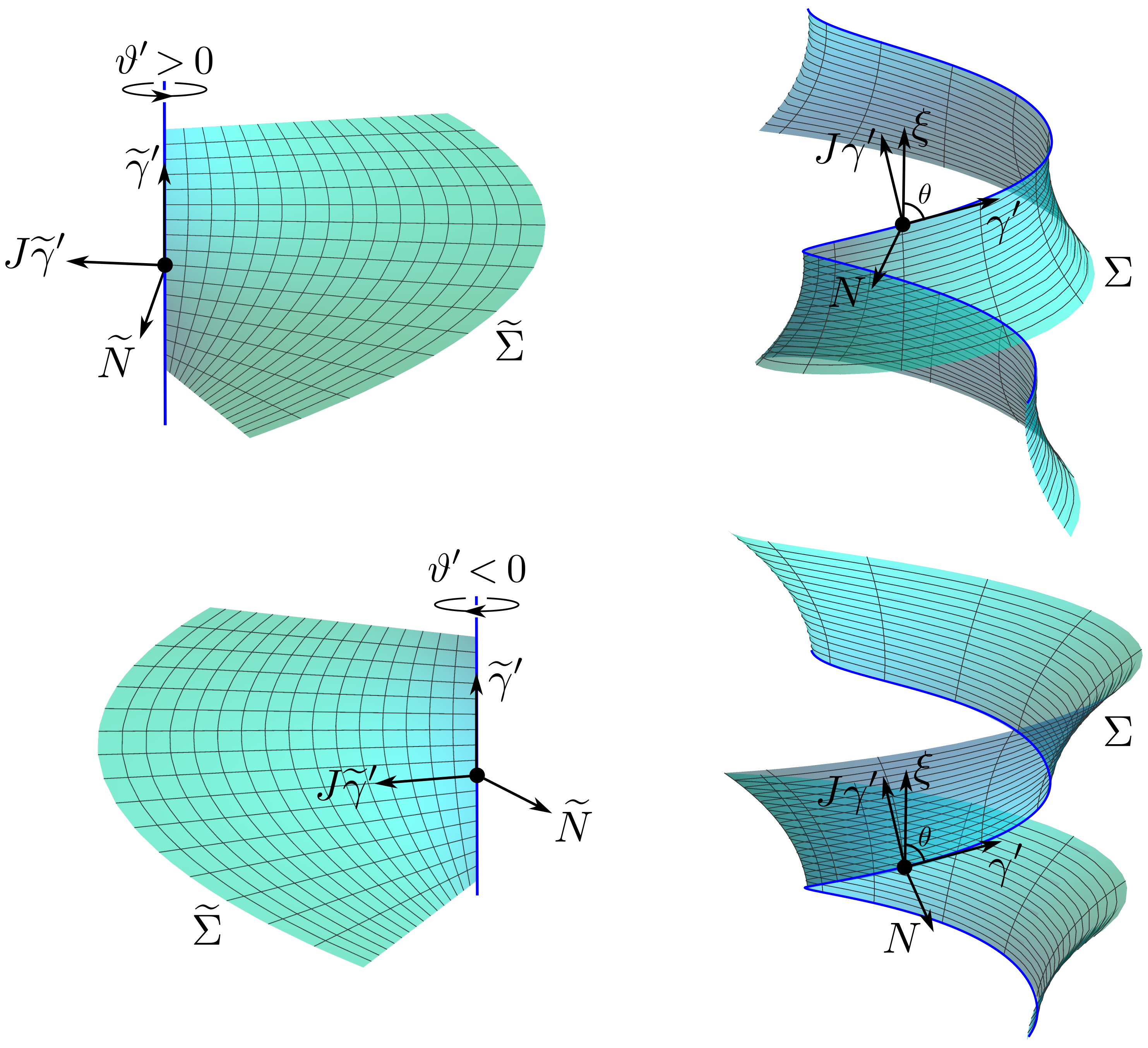}
\caption{Chosen orientation of a minimal surface $\widetilde\Sigma$ along a vertical geodesic $\widetilde\gamma$ and the corresponding curve $\gamma$ on its sister immersion $\Sigma$ according to the direction of rotation of $\widetilde N$.}\label{fig:orientation}
\end{center}\end{figure}

\begin{proof}
Since $\widetilde\gamma$ is vertical, we have $\nu=0$ along $\widetilde\gamma$, whence $\widetilde T=\widetilde\xi=\widetilde\gamma'$ along $\widetilde\gamma$ and $T=\xi$ along $\gamma$. Therefore, we get item (a) from the following computation:
\[\langle\gamma',\xi\rangle=\langle\gamma',T\rangle=\langle\widetilde\gamma',\cos(\theta)\widetilde T+\sin(\theta)J\widetilde T\rangle=\cos(\theta)\langle\widetilde\gamma',\widetilde\xi\rangle=\cos(\theta),\]
where we have used that $\langle\widetilde\gamma',J\widetilde T\rangle=0$ since $\widetilde\gamma'$ is vertical and $J\widetilde T$ is horizontal.

As for item (b), we will begin by differentiating~\eqref{lem:vertical-geodesics:eqn1} with respect to $\widetilde\gamma'$ using the ambient Levi-Civita connection in~\eqref{eqn:levi-civita}:
\begin{equation}\label{lem:vertical-geodesics:eqn2}
\begin{aligned}
\widetilde\nabla_{\widetilde\gamma'}\widetilde N&=-\vartheta'\sin(\vartheta) \widetilde E_1+\vartheta'\cos(\vartheta) \widetilde E_2+\cos(\vartheta)\widetilde\nabla_{\widetilde E_3}\widetilde E_1+\sin(\vartheta)\widetilde\nabla_{\widetilde E_3}\widetilde E_2\\
&=(\widetilde\tau-\vartheta')(\sin(\vartheta)\widetilde E_1-\cos(\vartheta)\widetilde E_2)=(\widetilde\tau-\vartheta')(\widetilde N\times\widetilde\gamma').
\end{aligned}\end{equation}
In particular, from~\eqref{lem:vertical-geodesics:eqn2} we infer that
\begin{equation}\label{lem:vertical-geodesics:eqn3}
\begin{aligned}
 \widetilde\tau-\vartheta'&=\langle \widetilde\nabla_{\widetilde\gamma'}\widetilde N,\widetilde N\times\widetilde\gamma'\rangle=-\langle\widetilde A\widetilde\gamma',J\widetilde\gamma'\rangle\\
 &=-\cos(\theta)\langle A\gamma'-H\gamma',J\gamma'\rangle+\sin(\theta)\langle JA\gamma'-HJ\gamma',J\gamma'\rangle\\
 &=-\cos(\theta)\langle A\gamma',J\gamma'\rangle+\sin(\theta)\langle A\gamma',\gamma'\rangle-H\sin(\theta).
\end{aligned}\end{equation}
We will next calculate the quantities $\langle A\gamma',J\gamma'\rangle$ and $\langle A\gamma',\gamma'\rangle$. To this end, observe that item (a) implies that $\gamma(t)=\phi_{t\cos(\theta)}(\widehat\alpha(t))$ up to a vertical translation, where $\{\phi_t\}_{t\in\R}$ is the $1$-parameter group of vertical translations and $\widehat\alpha$ is the horizontal lift of $\alpha=\pi\circ\gamma$ (i.e., $\gamma$ is a constant-angle lift of $\alpha$ as in~\cite[Eq.~(3-3)]{Man14}). This means that $\gamma'=\cos(\theta)\xi+\widehat\alpha'$, so we deduce that $\|\alpha'\|=\|\widehat\alpha'\|=\sin(\theta)$. Taking into account that $[\widetilde\alpha',\xi]=0$ and $\overline\nabla_{\xi}\xi=0$, we get that
\begin{equation}\label{lem:vertical-geodesics:eqn4}
\begin{aligned}
 \langle A\gamma',\gamma'\rangle=-\langle\overline\nabla_{\gamma'}N,\gamma'\rangle=\langle\overline\nabla_{\gamma'}\gamma',N\rangle&=\langle\overline\nabla_{\widehat\alpha'}\widehat\alpha',N\rangle+2\cos(\theta)\langle\overline\nabla_{\widehat\alpha'}\xi,N\rangle\\
 &=\sin^2(\theta)\kappa_g+2\tau\sin(\theta)\cos(\theta).
\end{aligned}\end{equation}
We have used that $\frac{1}{\sin(\theta)}\widehat\alpha'$ is horizontal and projects to $\frac{1}{\sin(\theta)}\alpha'$ (the unit tangent to $\alpha$) and $N$ projects to $\pi_*N$ (the unit normal to $\alpha$); since $\pi$ is a Riemannian submersion, we get that $\langle\overline\nabla_{\widehat\alpha'}\widehat\alpha',N\rangle=\sin^2(\theta)\kappa_g$. We also used that $\overline\nabla_{\widehat\alpha'}\xi=\tau\widehat\alpha\times\xi=\tau\sin(\theta)N$ because $\{\frac{1}{\sin(\theta)}\widehat\alpha',\xi,N\}$ is orthonormal and positively oriented. 

By the same argument, $J\gamma'=N\times(\widehat\alpha'+\cos(\theta)\xi)=\sin(\theta)\xi-\tfrac{\cos(\theta)}{\sin(\theta)}\widehat\alpha'$, whence
\begin{equation}\label{lem:vertical-geodesics:eqn5}
\begin{aligned}
\langle A\gamma',J\gamma'\rangle&=-\langle\overline\nabla_{\gamma'}N,J\gamma'\rangle=\langle\overline\nabla_{\gamma'}J\gamma',N\rangle
\\
&=\sin(\theta)\langle\overline\nabla_{\widehat\alpha'}\xi,N\rangle-\tfrac{\cos(\theta)}{\sin(\theta)}\langle\overline\nabla_{\widehat\alpha'}\widehat\alpha',N\rangle-\tfrac{\cos^2(\theta)}{\sin(\theta)}\langle\overline\nabla_{\widehat\alpha'}\xi,N\rangle\\
&=\tau(\sin^2(\theta)-\cos^2(\theta))-\sin(\theta)\cos(\theta)\kappa_g.
\end{aligned}
\end{equation}
Plugging~\eqref{lem:vertical-geodesics:eqn4} and~\eqref{lem:vertical-geodesics:eqn5} in~\eqref{lem:vertical-geodesics:eqn3}, we get $\widetilde\tau-\vartheta'=\tau\cos(\theta)-H\sin(\theta)+\kappa_g\sin(\theta)$, and the desired formula follows by the fact that $\widetilde\tau =\tau\cos(\theta)+H\sin(\theta)$.

As for item (c), observe that $\vartheta'>0$ (resp.\ $\vartheta'<0$) implies that $\kappa_g<2H$ (resp.\ $\kappa_g>2H$) by item (b). Therefore, if $\pi_*N$ points to the exterior (resp.\ interior) of 
$\pi(R)$ at $\alpha(p)$ for some $p$, then the vertical $H$-cylinder with normal $\gamma(p)$ lies locally outside
$\pi^{-1}(\pi(R))$ and is tangent to $\Sigma$ at $\gamma(p)$, in contradiction with the boundary maximum principle for $H$-surfaces.
\end{proof}

Although we have used the frame $\{\widetilde E_1,\widetilde E_2,\widetilde E_3\}$ to express the normal $\widetilde N$, it can be substituted with another positively oriented orthonormal frame $\{X_1,X_2,\widetilde\gamma'\}$ such that $X_1$ and $X_2$ project to constant vectors of $\mathbb{M}^2(\widetilde\kappa)$. We also remark that, if $\widetilde\gamma$ lies in the boundary of a region $R$ where $\nu<0$, item (c) still applies after some changes of sign by considering the continuation of $\widetilde\Sigma$ across $\widetilde\gamma$ by axial symmetry, so that the condition $\nu>0$ holds true in the symmetric region of $R$.

\begin{lemma}[Deformation of horizontal geodesics]\label{lem:horizontal-geodesics}
Let $\widetilde\gamma:[0,\ell]\to\widetilde\Sigma$ parametrize a horizontal geodesic with unit speed. We can write the normal of $\widetilde\Sigma$ along $\widetilde\gamma$ as
\begin{equation}\label{lem:horizontal-geodesics:eqn1}
\widetilde N_{\widetilde{\gamma}(t)}=\cos(\vartheta(t))(\widetilde\xi_{\widetilde{\gamma}(t)}\times\widetilde\gamma'(t))+\sin(\vartheta(t))\widetilde\xi_{\widetilde{\gamma}(t)},
\end{equation}
for some $\vartheta\in\mathcal C^\infty([0,\ell])$.
\begin{enumerate}[label=\emph{(\alph*)}]
    \item The conjugate curve $\gamma$ satisfies $\langle\gamma',\xi\rangle=\sin(\theta)\cos(\vartheta)$.
    \item The projection $\alpha=\pi\circ\gamma$ is a regular curve precisely at points where $\cos^2(\theta)+\nu^2\sin^2(\theta)\neq 0$, in which case $\alpha$ has geodesic curvature 
    \[\kappa_g=\frac{\left(2\tau(\cos^2(\theta)+\nu^2\sin^2(\theta))-\vartheta'\cos(\theta)\right)\sin(\theta)\cos(\vartheta)}{(\cos^2(\theta)+\nu^2\sin^2(\theta))^{3/2}}\]
    as a curve of $\mathbb{M}^2(\kappa)$ with respect to $\pi_*\frac{\xi\times\gamma'}{\|\xi\times\gamma'\|}$ as unit normal to $\alpha$.
    \item Denote by $P=\pi^{-1}(\alpha)$ the (flat) vertical cylinder over $\alpha$. At regular points of $\alpha$, the curve $\gamma$ has geodesic curvature
    \[\kappa_g^P=\frac{\vartheta'\sin(\vartheta)\sin(\theta)}{{(\cos^2(\theta)+\nu^2\sin^2(\theta))^{1/2}}}\]
    as a curve of $P$ with respect to $\frac{\gamma'\times(\xi\times\gamma')}{\|\xi\times\gamma'\|}$ as unit normal. Moreover, the cosine of the angle of the intersection $\Sigma\cap P$ along $\gamma$ equals $\frac{\cos(\theta)\cos(\vartheta)}{{(\cos^2(\theta)+\nu^2\sin^2(\theta))^{1/2}}}$.
\end{enumerate}
\end{lemma}

\begin{remark}
In Lemma~\ref{lem:horizontal-geodesics}, the angle function is given by $\nu=\sin(\vartheta)$ along $\widetilde\gamma$ or $\gamma$, whence $\sin(\vartheta)=\pm\sqrt{1-\nu^2}$. This means that the angle of rotation $\vartheta$ is equivalent to the angle function except for the fact that it captures an additional sign which is fundamental in the comprehension of the orientation of sister surfaces.
\end{remark}

\begin{proof}
In order to prove item (a), we will use the fact that $\widetilde\gamma$ is horizontal and hence orthogonal to $\widetilde T$, as well as~\eqref{lem:horizontal-geodesics:eqn1}. This implies that
\begin{equation}\label{lem:horizontal-geodesics:eqn1.5}
\begin{aligned}
\langle\gamma',\xi\rangle&=\langle\gamma',T\rangle=\langle\widetilde\gamma',\cos(\theta)\widetilde T+\sin(\theta)J\widetilde T\rangle=\sin(\theta)\langle\widetilde\gamma',J\widetilde T\rangle\\
&=\sin(\theta)\langle \widetilde\gamma',\widetilde N\times\widetilde\xi\rangle=\sin(\theta)\langle\widetilde N,\widetilde\xi\times\widetilde\gamma'\rangle=\sin(\theta)\cos(\vartheta).
\end{aligned}\end{equation}

Now we will deal with item (b). Assume for a moment that $\nu^2\neq 1$, so that $\{\frac{T}{\sqrt{1-\nu^2}},\frac{JT}{\sqrt{1-\nu^2}},N\}$ is a positively oriented orthonormal frame along $\gamma$. A computation similar to~\eqref{lem:horizontal-geodesics:eqn1.5} shows that $\langle\gamma',JT\rangle=\cos(\theta)\cos(\vartheta)$, whence
\begin{equation}\label{lem:horizontal-geodesics:eqn1.6}
\xi\times\gamma'\!=\!(T\!+\!\nu N)\!\times\!\left(\tfrac{\sin(\theta)\cos(\vartheta)}{1-\nu^2}T+\tfrac{\cos(\theta)\cos(\vartheta)}{1-\nu^2}JT\right)\!=\!\cos(\theta)\cos(\vartheta)N+\nu J\gamma'.
\end{equation}
This also holds true if $\nu^2=1$, in which case $\cos(\vartheta)=0$ and $\xi\times\gamma'=\pm J\gamma'$. The vector field $\xi\times\gamma'$ is horizontal and orthogonal to $\alpha=\pi\circ\gamma$. Moreover, the projection $\alpha$ is not regular precisely at points where $\xi\times\gamma'=0$ (i.e., where $\gamma'$ is vertical). Since
\begin{equation}\label{lem:horizontal-geodesics:eqn1.75}
\|\xi\times\gamma'\|^2=\cos^2(\theta)\cos^2(\vartheta)+\nu^2=\cos^2(\theta)(1-\nu^2)+\nu^2=\cos^2(\theta)+\nu^2\sin^2(\theta),
\end{equation}
the condition for the regularity of $\alpha$ in the statement follows immediately.

Now we aim at computing the geodesic curvature of $\alpha$. Since $\widetilde\gamma$ is an ambient geodesic, in particular we have that $\langle\widetilde\nabla_{\widetilde\gamma'}\widetilde\gamma',\widetilde N\rangle=0$. This means that
\begin{equation}\label{lem:horizontal-geodesics:eqn2}
\begin{aligned}
 0&=\langle\widetilde A\widetilde\gamma',\widetilde\gamma'\rangle=\cos(\theta)\langle A\gamma'-H\gamma',\gamma'\rangle-\sin(\theta)\langle JA\gamma'-HJ\gamma',\gamma'\rangle\\
 &=\cos(\theta)\langle A\gamma',\gamma'\rangle+\sin(\theta)\langle A\gamma',J\gamma'\rangle-H\cos(\theta).
\end{aligned}\end{equation}
Using a similar argument as in~\eqref{lem:vertical-geodesics:eqn2}, we get that $\widetilde\nabla_{\widetilde\gamma'}\widetilde N=-(\widetilde\tau+\vartheta')J\widetilde\gamma'$. (Actually, it can be useful to employ the Cartan model and assume that $\widetilde\gamma$ is the $x$-axis, which implies that $\widetilde\gamma'=\widetilde E_1$, $\widetilde N=\cos(\vartheta)\widetilde E_2+\sin(\vartheta)\widetilde E_3$ and $J\widetilde\gamma'=\widetilde N\times\widetilde\gamma'=\sin(\vartheta)\widetilde E_2-\cos(\vartheta)\widetilde E_3$.) Consequently, following the idea of~\eqref{lem:vertical-geodesics:eqn3}, we reach
\begin{equation}\label{lem:horizontal-geodesics:eqn3}
\begin{aligned}
\widetilde\tau+\vartheta'&=-\langle\widetilde\nabla_{\widetilde\gamma'}\widetilde N,J\widetilde\gamma'\rangle=\langle\widetilde A\widetilde\gamma',\widetilde J\widetilde\gamma'\rangle\\
&=\cos(\theta)\langle A\gamma'-H\gamma',J\gamma'\rangle-\sin(\theta)\langle JA\gamma'-HJ\gamma',J\gamma'\rangle\\
&=\cos(\theta)\langle A\gamma',J\gamma'\rangle-\sin(\theta)\langle A\gamma',\gamma'\rangle+H\sin(\theta).
\end{aligned}
\end{equation}

We can think of~\eqref{lem:horizontal-geodesics:eqn2} and~\eqref{lem:horizontal-geodesics:eqn3} as a system of linear equations, and solve
\begin{equation}\label{lem:horizontal-geodesics:eqn4}
\langle A\gamma',\gamma'\rangle=-\vartheta'\sin(\theta),\qquad 
\langle A\gamma',J\gamma'\rangle=(\widetilde\tau+\vartheta')\cos(\theta),
\end{equation}
where we also used that $\widetilde\tau\sin(\theta)=H$. This determines the shape operator along $\gamma$. Observe that $\overline\nabla_{\gamma'}\gamma'$ has no component in $\gamma'$ because $\gamma$ has unit speed, and it has no component in $J\gamma'$ either because $\gamma$ is a geodesic in $\Sigma$ (being a geodesic is an intrinsic property). All in all, from~\eqref{lem:horizontal-geodesics:eqn4}, we get that
\begin{equation}\label{lem:horizontal-geodesics:eqn5}
\overline\nabla_{\gamma'}\gamma'=\langle\overline\nabla_{\gamma'}\gamma',N\rangle N=\langle A\gamma',\gamma'\rangle N=-\vartheta'\sin(\theta)N.
\end{equation}

Write $\gamma'=\sigma\xi+\widehat\alpha'$ as in the proof of Lemma~\ref{lem:vertical-geodesics}, though in this case the coefficient $\sigma=\langle\gamma',\xi\rangle=\sin(\theta)\cos(\vartheta)$ is not constant. As $[\widehat\alpha',\xi]=0$ and $\overline\nabla_\xi\xi=0$, we get
\begin{equation}\label{lem:horizontal-geodesics:eqn6}
\begin{aligned}
\langle\overline\nabla_{\gamma'}\gamma',\xi\times\gamma'\rangle=\sigma'\langle\xi,\xi\times\gamma'\rangle+2\sigma\langle\overline\nabla_{\widehat\alpha'}\xi,\xi\times\gamma'\rangle+\langle\overline\nabla_{\widehat\alpha'}\widehat\alpha',\xi\times\gamma'\rangle.
\end{aligned}\end{equation}
Since $\overline\nabla_{\widehat\alpha'}\xi=\tau\widehat\alpha'\times\xi=\tau\gamma'\times\xi$ and $\langle\xi,\xi\times\gamma'\rangle=0$, we can use~\eqref{lem:horizontal-geodesics:eqn1.75} and~\eqref{lem:horizontal-geodesics:eqn6} in order to solve for $\langle\overline\nabla_{\widehat\alpha'}\widehat\alpha',\xi\times\gamma'\rangle$ in~\eqref{lem:horizontal-geodesics:eqn6}. This gives
\begin{equation}\label{lem:horizontal-geodesics:eqn7}
\begin{aligned}
\langle\overline\nabla_{\widehat\alpha'}\widehat\alpha',\xi\times\gamma'\rangle&=\left(-\cos(\theta)\vartheta'+2\tau(\cos^2(\theta)+\nu^2\sin^2(\theta))\right)\sin(\theta)\cos(\vartheta).
\end{aligned}
\end{equation}
Finally, we can take into account that $\widetilde\alpha'$ is horizontal, projects to $\alpha'$ and satisfies $\|\widehat\alpha'\|^2=1-\sigma^2=\cos^2(\theta)+\nu^2\sin^2(\theta)$, whereas $\xi\times\gamma'$ is horizontal and projects to the normal of $\alpha$. Dividing~\eqref{lem:horizontal-geodesics:eqn7} by $\|\widehat\alpha'\|^2\|\xi\times\gamma'\|$ and taking into account~\eqref{lem:horizontal-geodesics:eqn1.75}, we reach the desired formula for $\kappa_g$ in item (b).

As for item (c), observe that $\xi\times\gamma'$ is normal to $P$ and $\gamma'$ is a unit tangent to $P$ and $\gamma$, so $\eta=\frac{\gamma'\times(\xi\times\gamma')}{\|\xi\times\gamma'\|}$ is an unit normal to $\gamma$ as a curve of $P$. Using~\eqref{lem:horizontal-geodesics:eqn1.6} and the fact that $\{\gamma',J\gamma',N\}$ is orthonormal and positively oriented, we get that 
\begin{equation}\label{lem:horizontal-geodesics:eqn8}
\eta=\frac{-\cos(\theta)\cos(\vartheta)J\gamma'+\nu N}{\cos^2(\theta)+\nu^2\sin^2(\theta)}.
\end{equation}
Therefore, we can easily compute $\kappa_g^P=\langle\overline\nabla_{\gamma'}{\gamma'},\eta\rangle$ using~\eqref{lem:horizontal-geodesics:eqn5}, which leads to the formula in the statement. Finally, the cosine of the angle of intersection between $\Sigma$ and $P$ along $\gamma$ is the inner product of their unit normals, i.e., $\langle N,\frac{\xi\times\gamma'}{\|\xi\times\gamma'\|}\rangle$, so the result follows from~\eqref{lem:horizontal-geodesics:eqn1.6} and~\eqref{lem:horizontal-geodesics:eqn1.75}.
\end{proof}

\subsection{Spherical helicoids}\label{sec:spherical-helicoids}
The so-called spherical helicoids were introduced by Lawson~\cite{Law} in the round sphere and generalized to the rest of Berger spheres by Torralbo~\cite{Tor12}. They are minimal surfaces ruled by horizontal geodesics. In the Cartan model, they become Euclidean helicoids invariant by a $1$-parameter group of screw-motions (also isometries in the Berger metric), and can be described locally via the following parametrization that depends on a parameter $a\in\R$:
\begin{equation}\label{eqn:spherical-helicoids}
\widetilde Y_a(u,v)=\left(\tfrac{2}{\sqrt{\widetilde\kappa}}\tan(\tfrac{\sqrt{\widetilde\kappa}}{2}u)\cos(v),\tfrac{2}{\sqrt{\widetilde\kappa}}\tan(\tfrac{\sqrt{\widetilde\kappa}}{2}u)\sin(v),av\right),\end{equation}
where the curves $\widetilde h_v(u)=\widetilde Y_a(u,v)$ are horizontal geodesics and the parameter $v$ represents the action of the screw motion. It easily follows that $(R\circ \widetilde Y_a)(u,v)=\widetilde Y_{\frac{4\widetilde\tau}{\widetilde\kappa}-a}(\tfrac{\pi}{\sqrt{\widetilde\kappa}}-u,v)$,  
for all $u,v\in\R$ with $|u|<\frac{\pi}{\sqrt{\widetilde\kappa}}$, where 
\[R(x,y,z)=\left(\frac{4x}{\widetilde\kappa(x^2+y^2)},\frac{4y}{\widetilde\kappa(x^2+y^2)},-z+\tfrac{4\widetilde\tau}{\widetilde\kappa}\arg(x+iy)\right)\]
is a local expression (that depends on the choice of a continuous argument) of the axial symmetry about the horizontal geodesic 
\begin{equation}\label{eqn:symmetry-geodesic}
\widetilde h(t)=\left(\tfrac{2}{\sqrt{\widetilde\kappa}}\cos(t\sqrt{\widetilde\kappa}),\tfrac{2}{\sqrt{\widetilde\kappa}}\sin(t\sqrt{\widetilde\kappa}),\tfrac{2\widetilde\tau}{\sqrt{\widetilde\kappa}}t\right).
\end{equation}
This means that $\widetilde\Sigma_a$ and $\widetilde\Sigma_{\frac{4\widetilde\tau}{\widetilde\kappa}-a}$ are actually the same surface and reflects the fact that each spherical helicoid has actually two vertical axes that project to antipodal points of $\mathbb{S}^2(\kappa)$ with different speeds of rotation:
\begin{itemize}
  \item Let $\widetilde\gamma_+(t)=\widetilde Y_a(0,\frac{t}{a})$ be the $z$-axis parametrized with unit speed. We have $\widetilde N_{\widetilde\gamma(t)}=\sin(\frac{1}{a}t)\widetilde E_1-\cos(\frac{1}{a}t)\widetilde E_2$ so we can choose $\vartheta_+(t)=\frac{1}{a}t-\frac{\pi}{2}$
  \item The antipodal axis $\widetilde\gamma_-(t)=\widetilde Y_{\frac{4\widetilde\tau}{\widetilde\kappa}-a}(0,\frac{-\widetilde\kappa}{4\widetilde\tau-\widetilde\kappa a}t)$ gives $\vartheta_-(t)=\frac{-\widetilde\kappa}{4\widetilde\tau-\widetilde\kappa a}t+\frac{\pi}{2}$ (we have a change of sign because $R$ maps $\xi$ to $-\xi$).
\end{itemize}
By Lemma~\ref{lem:vertical-geodesics}, we deduce that the conjugate curves $\gamma_{\pm}$ are Euclidean helices in the Cartan model and project to curves of the same constant geodesic curvature if and only if $a=\frac{2\widetilde\tau}{\widetilde\kappa}$. This means that this is the only helicoid in which we can expect to recover the $H$-tubes as sister surfaces.

\begin{remark}
The angle function in the parametrization~\eqref{eqn:spherical-helicoids} reads
\[\nu(u,v)=\tfrac{\sqrt{2\widetilde\kappa} \sin \left(\sqrt{\widetilde\kappa } u\right)}{\sqrt{2 a^2 \widetilde\kappa ^2-8 a \widetilde\kappa  \widetilde\tau +8 \widetilde\tau  (a \widetilde\kappa -2 \widetilde\tau ) \cos
   \left(\sqrt{\widetilde\kappa } u\right)+\widetilde\kappa +12 \widetilde\tau ^2-\left(\widetilde\kappa -4 \widetilde\tau ^2\right) \cos \left(2 \sqrt{\widetilde\kappa } u\right)}},\]
which takes the value $1$ if and only if $0\leq a\leq \frac{4\widetilde\tau}{\widetilde\kappa}$. The extremal cases $a=0$ and $a=\frac{4\widetilde\tau}{\widetilde\kappa}$ give the minimal spheres in $\E(\widetilde\kappa,\widetilde\tau)$, whereas $a=\frac{2\widetilde\tau}{\widetilde\kappa}$ gives the minimal torus $T_0$, see Remark~\ref{rmk:clifford}. This torus is the only spherical helicoid that contains an additional horizontal geodesic (different from the rulings $\widetilde h_v$), namely $\widetilde h$.

We cannot expect that the sister surfaces of $\widetilde Y_a$ yield $H$-tubes around a non-horizontal geodesic either. This is evidenced by the results in~\cite[\S4]{MT14} that the conjugate surfaces ($\theta=\frac\pi2$) of $\widetilde Y_a$ in $\E(4H^2+\kappa,H)$ are rotational nodoids in $\mathbb{M}^2(\kappa)\times\R$, which do not close their periods unless $a=\frac{2\widetilde\tau}{\widetilde\kappa}$.
\end{remark}

\begin{theorem}\label{thm:spherical-helicoids}
The family of horizontal $H$-tubes is preserved by the sister correspondence. Moreover, the correspondence induces local isometries between associate $H$-tubes which are not global (as explained in the caption of Figure~\ref{fig:conformal-type}).
\end{theorem}

\begin{figure}
\begin{center}
\includegraphics[width=0.65\textwidth]{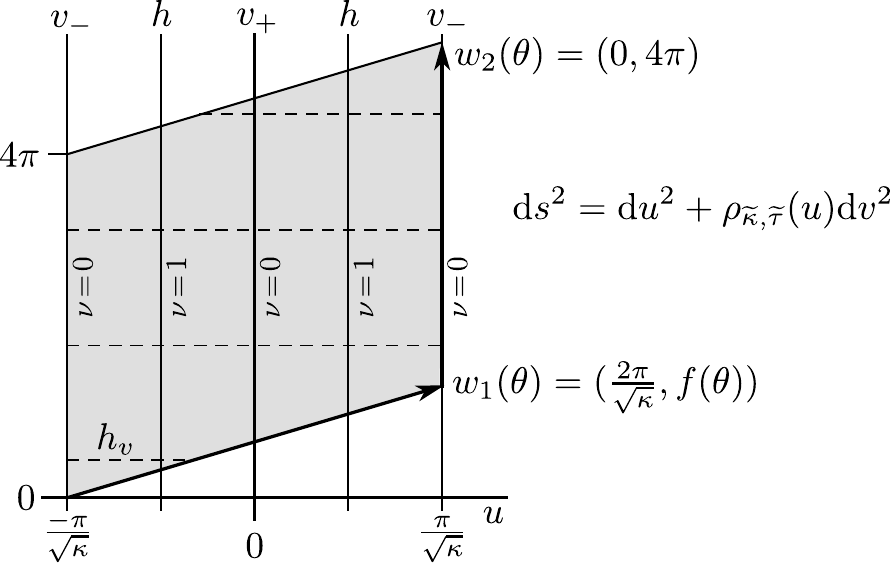}
\caption{Fundamental region for the $H$-torus as the quotient of $(\R^2,\df s^2)$ by the lattice spanned by $\{w_1(\theta),w_2(\theta)\}$ (we are assuming $\kappa>0$). The dashed horizontal lines represent a single curve $h_v$ coming from a ruling of the associate minimal helicoid. This curve spirals through the torus $T_H$, being generically dense on it.}\label{fig:conformal-type}
\end{center}
\end{figure}

\begin{proof}
Let $\widetilde\Sigma$ be the minimal torus in $\E(\widetilde\kappa,\widetilde\tau)$ locally parametrized by $\widetilde Y_{\frac{2\widetilde\tau}{\widetilde\kappa}}$, with $\widetilde\kappa,\widetilde\tau>0$. Given $\theta\in[0,\pi)$, we will consider the sister immersion $\Sigma\subset\E(\kappa,\tau)$ with phase angle $\theta$, so that $\kappa=\widetilde\kappa-4\widetilde\tau^2\sin^2(\theta)$, $\tau=\widetilde\tau\cos(\theta)$ and $H=\widetilde\tau\sin(\theta)$. Define $\gamma_+$, $\gamma_-$, $h_v$ ($v\in\R$) and $h$, as the counterpart curves in $\Sigma$ to those in $\widetilde\Sigma$. 

By Lemma~\ref{lem:horizontal-geodesics} since $\nu=1$ along $\widetilde h$, the angle of rotation satisfies $\cos(\vartheta)=0$, whence $\langle h',\xi\rangle=0$ and $h$ projects to a geodesic, so that $h$ is a horizontal geodesic, see~\cite[Prop.~3.6]{Man14}. Taking into account~\cite[Prop.~1]{CMT}, we have that $\Sigma$ is invariant by translations along $h$. Since $\widetilde\Sigma$ is invariant by axial symmetries about vertical geodesics intersecting $\widetilde h$ ($\widetilde\Sigma$ is the union of all horizontal geodesics orthogonal to $\widetilde h$), it follows that $\Sigma$ is invariant by axial symmetries about vertical geodesics intersecting $h$ (note that such axial symmetries preserve the ambient orientation and the orientation of vertical fibers, and hence the fundamental data). Likewise, $\Sigma$ is invariant by axial symmetries about horizontal geodesics intersecting $v_+$ or $v_-$ because $\widetilde\Sigma$ is invariant by axial symmetries about horizontal geodesics intersecting $\widetilde v_+$ or $\widetilde v_-$ (again, the fundamental data of these transformations agree, see~\cite[Rmk.~2]{CMT}). All in all, there is a horizontal geodesic $\Gamma$ that differs from $h$ in a vertical translation and such that $\Sigma$ is invariant by all symmetries leaving $\Gamma$ invariant. By Lemma~\ref{lemma:existence}, $\Sigma$ must be the $H$-tube around $\Gamma$ in $\E(\kappa,\tau)$ up to a vertical translation.

In the parametrization $\widetilde Y_{\frac{2\widetilde\tau}{\widetilde\kappa}}$ given by~\eqref{eqn:spherical-helicoids}, the metric of $\Sigma$ reads
\begin{equation}\label{thm:spherical-helicoids:eqn1}
\df s^2=\df u^2+\rho_{\widetilde\kappa,\widetilde\tau}(u)\df v^2,\quad \text{where }\rho_{\widetilde\kappa,\widetilde\tau}(u)=\tfrac{1}{\widetilde\kappa}\sin^2(u\sqrt{\widetilde\kappa})+\tfrac{4\widetilde\tau^2}{\widetilde\kappa^2}\cos^2(u\sqrt{\widetilde\kappa}).
\end{equation}
This actually defines a metric in all $\R^2$, where the coordinate curves with constant $v$ represent the horizontal rulings $\widetilde h_{v}$. Since each $\widetilde h_{v}$ projects two-to-one onto a great circle of $\mathbb{S}^2(\kappa)$, it easily follows that $\widetilde h_{v}$ intersects each orbit of the $1$-parameter group of isometries twice. This implies that $\widetilde\Sigma$ is globally isometric to the quotient of $(\R^2,\df s^2)$ by the lattice spanned by $w_1(0)=(\frac{2\pi}{\sqrt{\widetilde\kappa}},2\pi)$ and $w_2(0)=(0,4\pi)$.

The sister surface $\Sigma$ in $\E(\kappa,\tau)$ with phase angle $\theta$ has the same universal cover $(\R^2,\df s^2)$, the same parameter $v$ representing the isometric invariance (see~\cite[Prop.~1]{CMT}) and the same periodicity in the variable $u$. If $\kappa=\widetilde\kappa-4\widetilde\tau^2\sin^2(\theta)>0$, then $\Sigma$ is globally isometric to the quotient of $(\R^2,\df s^2)$ by the lattice spanned by $w_1(\theta)=(\frac{2\pi}{\sqrt{\widetilde\kappa}},b_{\widetilde\kappa,\widetilde\tau}(\theta))$ and $w_2(\theta)=(0,4\pi)$ for some function $b_{\widetilde\kappa,\widetilde\tau}(\theta)$ with $b_{\widetilde\kappa,\widetilde\tau}(0)=2\pi$, see Figure~\ref{fig:conformal-type}. If $\kappa\leq 0$, then $\Sigma$ is a cylinder and the periodicity is given just by $w_1(\theta)$.

We will obtain an expression for $b_{\widetilde\kappa,\widetilde\tau}(\theta)$ by means of $X(\varphi,v)$ given by~\eqref{eqn:X:berger},~\eqref{eqn:X:nil} or~\eqref{eqn:X:SL} after the substitutions~\eqref{eqn:parametrization:berger},~\eqref{eqn:parametrization:Nil} and~\eqref{eqn:parametrization:SL}. Observe that the angle function is a function of $\varphi$ and the curves $h_v$ are orthogonal to the level curves of the angle function in view of~\eqref{thm:spherical-helicoids:eqn1} (this condition is intrinsic). Write $h_v(\varphi)=X(\varphi,v(\varphi))$ with $\varphi\in\R$, whence the orthogonality condition reads $\langle X_\varphi+v'X_v,X_v\rangle=0$. Regardless the sign of $\kappa$, this last equation can be expressed as
\begin{equation}\label{thm:spherical-helicoids:eqn2}
\begin{aligned}
v'(\varphi)&=\frac{-\langle X_\varphi,X_v\rangle}{\langle X_v,X_v\rangle}=\frac{\tau\sqrt{\kappa}\cos^2(\varphi)}{\sqrt{4H^2+\kappa\cos^2(\varphi)}\sqrt{H^2+\tau^2\cos^2(\varphi)}}
\\&=\frac{\cos(\theta)\sqrt{\widetilde\kappa-4\widetilde\tau^2\sin^2(\theta)}\cos^2(\varphi)}{\sqrt{\widetilde\kappa\cos^2(\varphi)+4\widetilde\tau^2\sin^2(\theta)\sin^2(\varphi)}\sqrt{\cos^2(\varphi)+\sin^2(\theta)\sin^2(\varphi)}}.
\end{aligned}\end{equation}
Observe that $\Sigma$ is recovered in the parametrization $X(\varphi,v)$ as the quotient of $\R^2$ by the (rectangular) lattice spanned by $(2\pi,0)$ (and $(0,4\pi)$ if $\kappa>0$). By identifying the curves $h_v$ in the parametrizations $X$ and $\widetilde Y_{\frac{2\widetilde\tau}{\widetilde\kappa}}$ (recall that they have constant $v$ in the latter case), we get $b_{\widetilde\kappa,\widetilde\tau}(\theta)=v(2\pi)-v(0)$. From~\eqref{thm:spherical-helicoids:eqn2}, we deduce that
\begin{equation}\label{thm:spherical-helicoids:eqn3}
b_{\widetilde\kappa,\widetilde\tau}(\theta)=\int_0^{2\pi}\!\!\!\tfrac{\cos(\theta)\sqrt{\widetilde\kappa-4\widetilde\tau^2\sin^2(\theta)}\cos^2(\varphi)}{\sqrt{\widetilde\kappa\cos^2(\varphi)+4\widetilde\tau^2\sin^2(\theta)\sin^2(\varphi)}\sqrt{\cos^2(\varphi)+\sin^2(\theta)\sin^2(\varphi)}}\df\varphi.
\end{equation}
Although this integral seems to be nonexplicit, the integrand is clearly a strictly decreasing function of $\theta$ in the range $\theta\in[0,\frac\pi2]$ with $b_{\widetilde\kappa,\widetilde\tau}(\frac\pi2)=0$. Moreover, we have the symmetry $b_{\widetilde\kappa,\widetilde\tau}(\pi-\theta)=-b_{\widetilde\kappa,\widetilde\tau}(\theta)$ for all $\theta$, so we get that $\theta\mapsto b_{\widetilde\kappa,\widetilde\tau}(\theta)$ decreases from $2\pi$ to $-2\pi$ as $\theta$ runs from $0$ to $\pi$.
\end{proof}

We have shown that all $H$-tubes admit a warped-product metric given by~\eqref{thm:spherical-helicoids:eqn1}. This metric can be easily transformed into a conformal metric by a change of coordinates $u=g(s)$. The conformal condition is equivalent to the following \textsc{ode}:
\begin{equation}\label{eqn:conformal-factor}
g'(s)^2=\rho_{\widetilde\kappa,\widetilde\tau}(g(s))=\frac{\widetilde\kappa\sin^2(g(s)\sqrt{\widetilde\kappa})+4\widetilde\tau^2\cos^2(g(s)\sqrt{\widetilde\kappa})}{\widetilde\kappa^2},\quad g(0)=0.
\end{equation}
The right-hand side stays always between two positive constants, so the solution to~\eqref{eqn:conformal-factor} is a strictly increasing diffeomorphism from $\R$ to $\R$. There is a unique $a_{\widetilde\kappa,\widetilde\tau}\in\R$ such that $g(a_{\widetilde\kappa,\widetilde\tau})=\frac{2\pi}{\sqrt{\widetilde\kappa}}$, whence we deduce from~\eqref{eqn:conformal-factor} that
\begin{equation}\label{eqn:lattice-first}
g(s+a_{\widetilde\kappa,\widetilde\tau})=g(s)+\tfrac{2\pi}{\sqrt{\widetilde\kappa}},\qquad g'(s+a_{\widetilde\kappa,\widetilde\tau})=g'(s),\qquad\text{for all }s\in\R.
\end{equation}
Since the change $u=g(s)$ agrees with the periodicity of the tangent function in~\eqref{eqn:spherical-helicoids}, the conformal metric $g'(s)^2(\df s^2+\df v^2)$ is well defined in the quotient.

\begin{corollary}\label{coro:conformal}
The sister $H$-torus of the minimal torus in $\E(\widetilde\kappa,\widetilde\tau)$ with phase $\theta\in\R$ (and $\kappa>0$) is conformally the quotient of $\mathbb{R}^2$ by the lattice spanned by $(a_{\widetilde\kappa,\widetilde\tau},b_{\widetilde\kappa,\widetilde\tau}(\theta))$ and $(0,4\pi)$, where $a_{\widetilde\kappa,\widetilde\tau}$ and $b_{\widetilde\kappa,\widetilde\tau}(\theta)$ are defined by~\eqref{thm:spherical-helicoids:eqn3} and~\eqref{eqn:lattice-first}.
\end{corollary}

\begin{remark}\label{rmk:conformal}
The proposed lattice can be normalized by a similarity to $(1,0)$ and $(\frac{1}{2\pi}b_{\widetilde\kappa,\widetilde\tau}(\theta),\frac{1}{2\pi}a_{\widetilde\kappa,\widetilde\tau})$. In the proof of Theorem~\ref{thm:spherical-helicoids}, we have seen that $\theta\mapsto\frac{1}{2\pi}b_{\widetilde\kappa,\widetilde\tau}(\theta)$ ranges from $\frac{-1}{2}$ to $\frac{1}{2}$ for $\theta\in\R$ and fixed $\widetilde\kappa$ and $\widetilde\tau$. On the other hand, Equation~\eqref{eqn:conformal-factor} implies that $a_{\widetilde\kappa,\widetilde\tau}$ is arbitrarily large (resp.\ close to zero) if both $\frac{1}{\widetilde\kappa}$ and $\frac{4\widetilde\tau^2}{\widetilde\kappa^2}$ are close to zero (resp.\ are large enough). This means that $a_{\widetilde\kappa,\widetilde\tau}$ takes all positive real values.

By looking at the moduli space of conformal tori (e.g., see~\cite[\S3]{Pinkall}), we deduce that all compact genus-one Riemann surfaces can be immersed as an horizontal $H$-tube in some $\E(\kappa,\tau)$.
\end{remark}

\begin{remark}
The function $g$ in~\eqref{eqn:conformal-factor} can be integrated explicitly in terms of elliptic functions. More precisely, the software \emph{Mathematica} gives 
\[g(s)=\tfrac{1}{\sqrt{\kappa}}\am(\tfrac{2\tau}{\sqrt{\kappa}}s,1-\tfrac{\kappa}{4\tau^2}),\]
where $\am(x,m)$ denotes the Jacobi amplitude of modulus $m$. For further reference, this function can be looked up at \href{https://mathworld.wolfram.com/JacobiAmplitude.html}{https://mathworld.wolfram.com}.
\end{remark}

\subsection{The isoperimetric profile of $H$-tubes}
Torralbo and Urbano proved that in some ranges of $\kappa,\tau>0$ the least area surface spanning a given volume in the Berger sphere $\E(\kappa,\tau)$ is either an $H$-sphere or an embedded $H$-torus~\cite{TU}. We will show numerically that the $H$-tubes can be discarded in this competition since, for a fixed volume, they have larger area than $H$-spheres if $\kappa-4\tau^2>0$ and have larger area than Hopf $H$-tori is $\kappa-4\tau^2<0$. We will fix $\kappa=4$ henceforth after a suitable rescaling, so formulas for the area and volume become more tractable.

On the one hand, the area of the $H$-tube $T_H$ in $\E(\kappa,\tau)$ can be obtained as the integral of the area element of the parametrization~\eqref{eqn:parametrization:berger}, which gives
\[\Area(T_H)=H\pi\int_0^{2\pi}\frac{\sqrt{H^2+\tau^2\cos^2(u)}}{(H^2+\cos^2(u))^{3/2}}\df u.\]
On the other hand, we can take advantage of the foliation properties given by Theorem~\ref{thm:foliation} because we can consider $H$ as an additional parameter to parametrize the interior of $T_H$. Even in the event that there is no such foliation (i.e., if $(1-x_0^2)\kappa-4\tau^2>0$), the algebraic volume still gives the desired value. We can take the volume element of the parametrization~\eqref{eqn:parametrization:berger} with $w=H$ as third parameter to get (after a really cumbersome calculation) that
\begin{align*}
\Vol(T_H)&=\tfrac{\pi}{2}\int_H^{+\infty}\int_0^{2\pi}\frac{w\sqrt{w^2+\tau^2\cos^2(u)}}{(1+w^2)(w^2+\cos^2(u))^{3/2}}\,\df u\,\df w\\
&\quad+\tfrac{\pi\sqrt{1-\tau^2}}{2}\int_H^{+\infty}\int_0^{2\pi}\frac{w^2\arctanh(\frac{w\sin(u)}{\sqrt{1+w^2}\sqrt{w^2+\tau^2\cos^2(u)}})\sin(u)}{(1+w^2)^{3/2}(w^2+\cos^2(u))^{3/2}}\,\df u\,\df w.
\end{align*}
However, since the ambient Berger sphere is compact, one has to take into account the complementary region, and consider $\Vol(\E(\kappa,\tau))-\Vol(T_H)=\frac{32\tau\pi^2}{\kappa^2}-\Vol(T_H)$ also as a volume demarcated by $T_H$. It seems that $H\mapsto\Area(T_H)$ and $H\mapsto\Vol(T_H)$ cannot be inverted by elementary methods, so we cannot compare directly the $H$-tubes and the other candidates to solve the isoperimetric problem. In spite of this technical inconvenience, we have plotted numerically in Figure~\ref{fig:profiles} the curve $H\mapsto (\Vol(T_H),\Area(T_H))$, as well as the corresponding curves for $H$-spheres and Hopf $H$-tori using the formulas given by Torralbo~\cite{Tor10}. Recall that in the case of the round sphere ($\tau=1$), the red and green curves coincide because the Hopf $H$-tori and the $H$-tubes are congruent.

\begin{figure}
\includegraphics[width=0.3\linewidth]{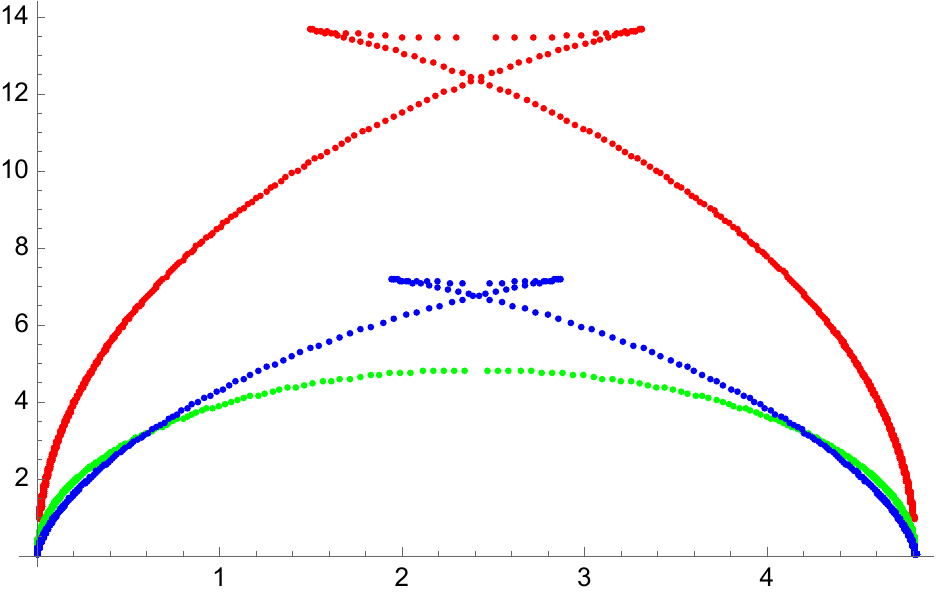}
\includegraphics[width=0.3\linewidth]{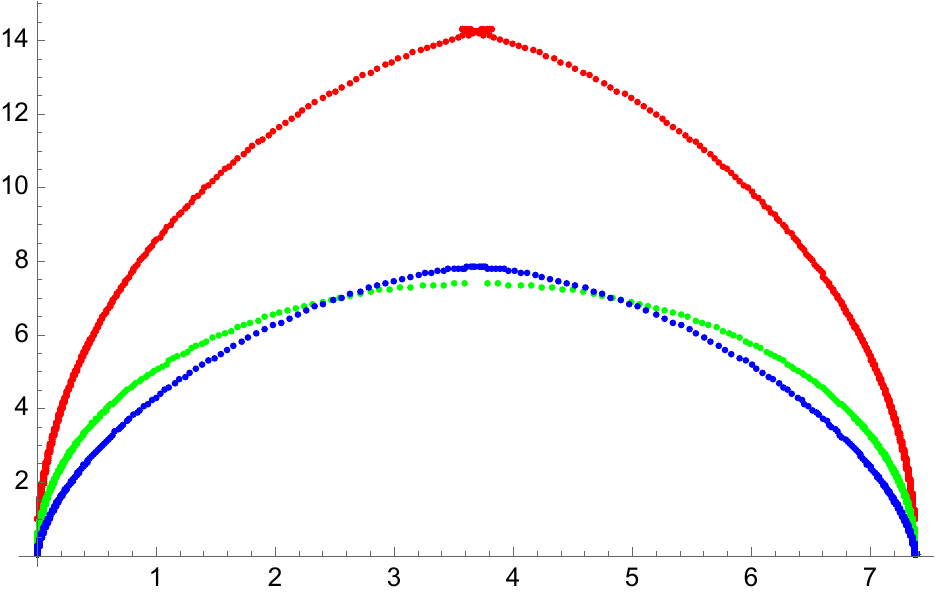}\\
\includegraphics[width=0.3\linewidth]{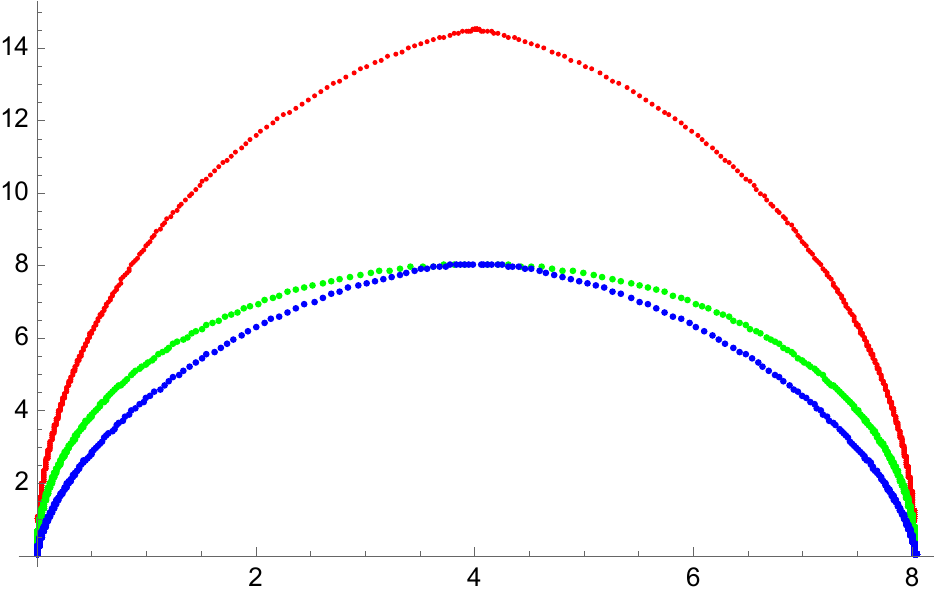}
\includegraphics[width=0.3\linewidth]{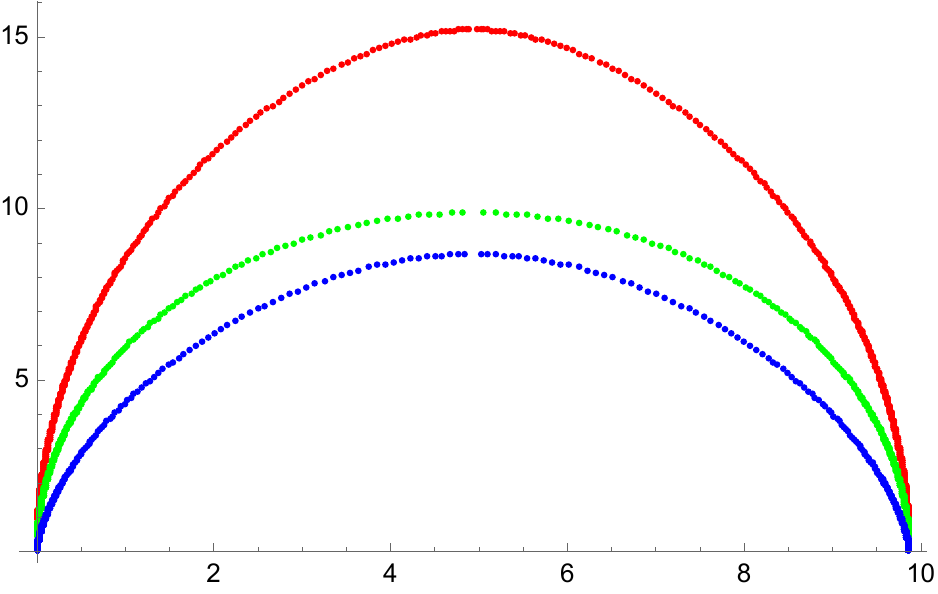}\\
\includegraphics[width=0.3\linewidth]{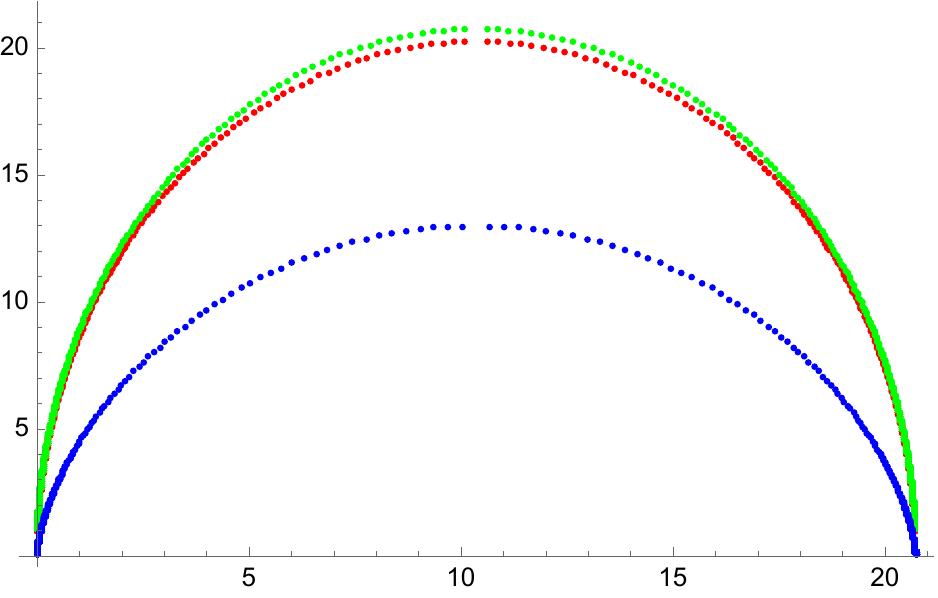}
\includegraphics[width=0.3\linewidth]{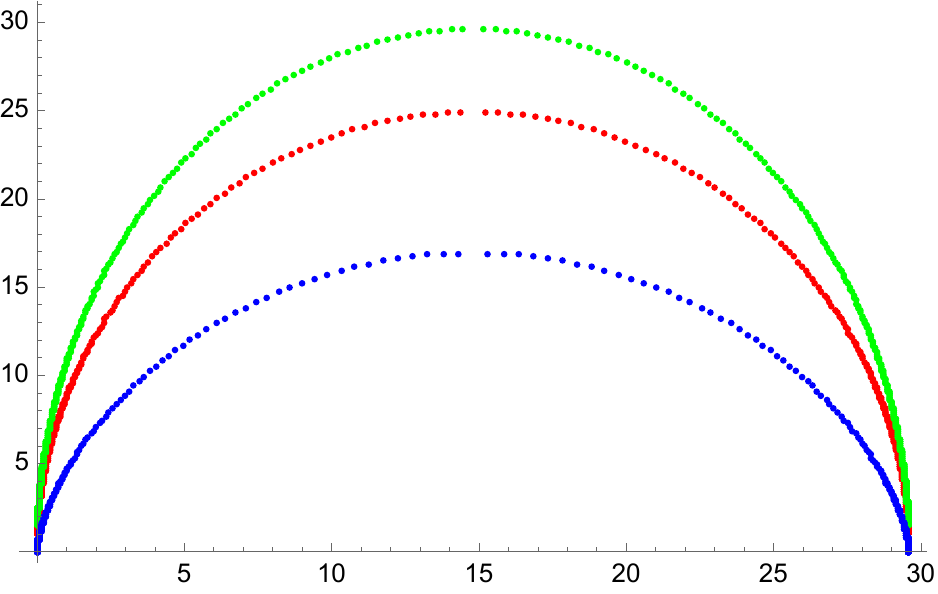}\\
\includegraphics[width=0.3\linewidth]{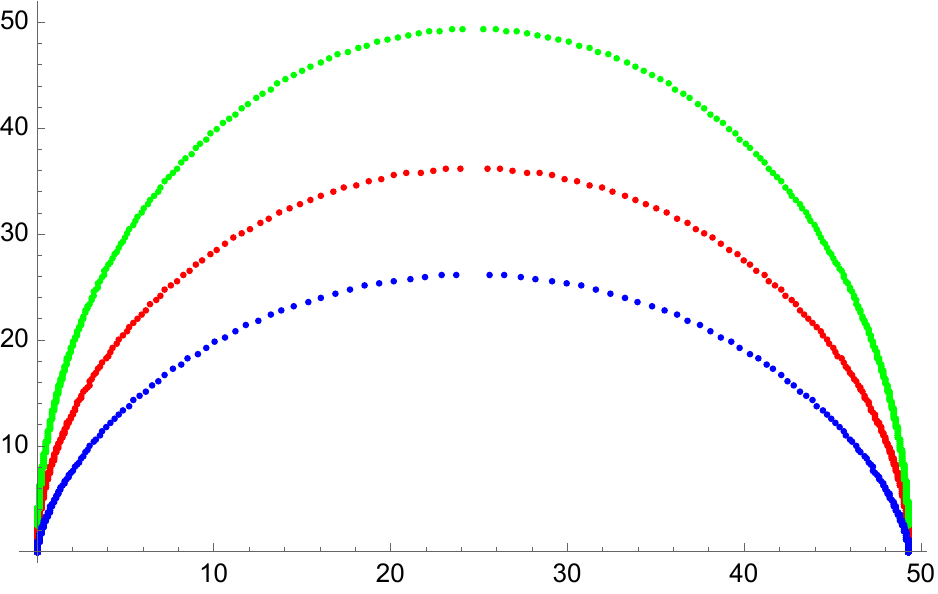}
\includegraphics[width=0.3\linewidth]{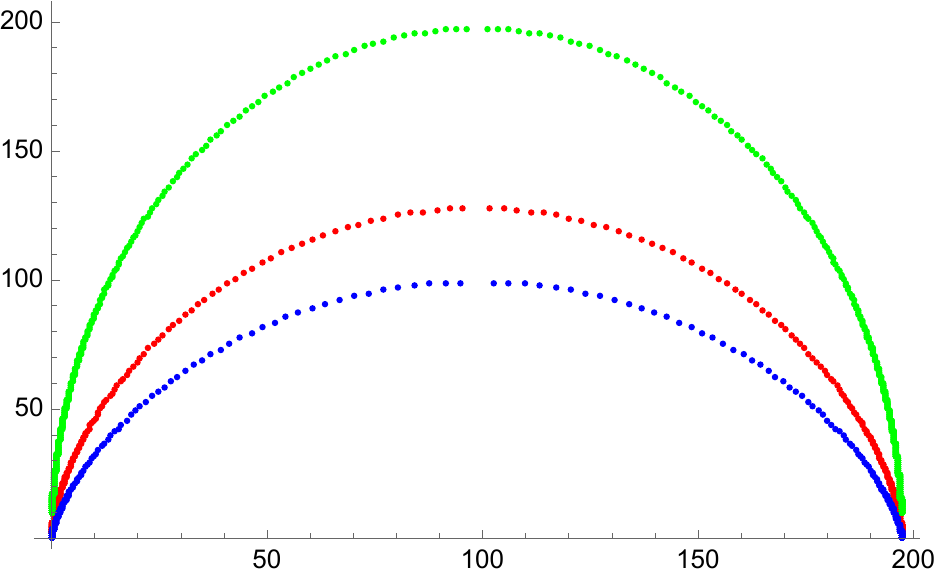}
\caption{Numerical plot of the isoperimetric profiles of $H$-spheres (blue), Hopf $H$-tori (green) and horizontal $H$-tubes (red) in Berger spheres $\E(4,\tau)$ where $H$ runs from $0.025$ to $20$ with step $0.025$ and $\tau\in\{0.244,0.374,0.407,0.5,1.05,1.5,2.5,10\}$. The horizontal axis is the volume enclosed by the surface and the vertical axis represents its area.}\label{fig:profiles}
\end{figure}



\begin{thebibliography}{00}

\bibitem{CMT} J.\ Castro-Infantes, J.\ M.\ Manzano, F.\ Torralbo.
\newblock Conjugate Plateau constructions in product spaces.
\newblock To appear as a book chapter in the Springer-RSME series.  Preprint available at \href{https://arxiv.org/abs/2203.13162}{arXiv:2203.13162}.

\bibitem{Dan} B.\ Daniel.
\newblock Isometric immersions into 3-dimensional homogeneous manifolds.
\newblock \emph{Comment.\ Math.\ Helv.} \textbf{82} (2007), no.\ 1, 87--131.

\bibitem{DH09} B.\ Daniel, L.\ Hauswirth.
\newblock Half-space theorem, embedded minimal annuli and minimal graphs in the Heisenberg group.
\newblock \emph{Proc.\ London Math.\ Soc.} \textbf{98} (2009), no.\ 3, 445--470.

\bibitem{FMP} C.\ Figueroa, F.\ Mercuri, R.\ H.\ L.\ Pedrosa.
\newblock Invariant surfaces of the Heisenberg groups.
\newblock \emph{Ann.\ Mat.\ Pura Appl.}, \textbf{177} (1999), no.\ 4, 173--194. 

\bibitem{Kase} P.\ Käse.
\newblock Screw motion surfaces of constant mean curvature in homogeneous 3-manifolds.
\newblock Preprint available at \href{https://arxiv.org/abs/2203.10968}{arXiv:2203.10968}.

\bibitem{Law} H.\ B.\ Lawson.
\newblock Complete Minimal Surfaces in $S^3$.
\newblock \emph{Ann. of Math. (2)}, \textbf{92} (1970), no.\ 3, 335--374.

\bibitem{Man13} J.\ M.\ Manzano.
\newblock Estimates for constant mean curvature graphs in $M\times\mathbb{R}$.
\newblock \emph{Rev.\ Mat.\ Iberoam.} \textbf{29} (2013), no.\ 4, 1263--1281.

\bibitem{Man14} J.\ M.\ Manzano.
\newblock On the classification of Killing submersions and their isometries.
\newblock \emph{Pac.\ J.\ Math.} \textbf{270} (2014), no.\ 2, 367--692.

\bibitem{MT14} J.\ M.\ Manzano, F.\ Torralbo.
\newblock New examples of constant mean curvature surfaces in $\mathbb{S}^2\times\mathbb{R}$ and $\mathbb{H}^2\times\mathbb{R}$.
\newblock \emph{Michigan Math.\ J.} \textbf{63} (2014), no.\ 4, 701--723.

\bibitem{MT22} J.\ M.\ Manzano, F.\ Torralbo.
\newblock Horizontal Delaunay surfaces with constant mean curvature in $\mathbb{S}^2\times\mathbb{R}$ and $\mathbb{H}^2\times\mathbb{R}$.
\newblock \emph{Cambridge J. Math.} \textbf{10} (2022), no.\ 3, 657--688.

\bibitem{Mazet2013}
L.\ Mazet. 
\newblock A general halfspace theorem for constant mean curvature surfaces.
\newblock \emph{Amer.\ J.\ Math.} \textbf{125} (2013), no.~3, 801--834. 

\bibitem{HM90}
D.\ Hoffman, W.H. Meeks.
\newblock The strong halfspace theorem for minimal surfaces.
\newblock \emph{Invent.\ Math.} \textbf{101} (1990), 373--377.




\bibitem{Onnis}
I.\ I.\ Onnis.
\newblock Invariant surfaces with constant mean curvature in
  {$\mathbb{H}^2\times \mathbb{R}$}.
\newblock \emph{Ann. Mat. Pura Appl.}, \textbf{187} (2008), no.\ 4, 667--682.

\bibitem{Pena} C.\ Peñafiel.
\newblock Invariant surfaces in $\widetilde{\mathrm{PSL}}_2(\R,\tau)$ and applications.
\newblock \emph{Bull.\ Braz.\ Math.\ Soc.}, \textbf{43} (2012), no.\ 4, 545--578.

\bibitem{Pedrosa}R.\ Pedrosa.
\newblock The isoperimetric problem in spherical cylinders.
\newblock \emph{Ann. Glob. Anal. Geom.}, \textbf{26} (2004), no.\ 4, 333--354.

\bibitem{Pinkall}U.\ Pinkall.
\newblock Hopf tori in $S^3$. 
\newblock \emph{Invent.\ Math.} \textbf{81} (1985), 379--386.

\bibitem{Ple14} J.\ Plehnert.
\newblock Constant mean curvature $k$-noids in homogeneous manifolds.
\newblock \emph{Illinois J.\ Math.} \textbf{58} (2014), no.\ 1, 233--249.

\bibitem{Tor10} F.\ Torralbo. 
\newblock Rotationally invariant constant mean curvature surfaces oin homogeneous $3$-manifolds.
\newblock \emph{Differential Geom.\ Appl.} \textbf{28} (2010), 593--607.

\bibitem{Tor12} F.\ Torralbo. 
\newblock Compact minimal surfaces in the Berger spheres.
\newblock \emph{Ann.\ Global Anal.\ Geom.} \textbf{41} (2012), 391--405.

\bibitem{TU} F.\ Torralbo, F.\ Urbano. 
\newblock Compact stable constant mean curvature surfaces in homogeneous 3-manifolds.
\newblock \emph{Indiana Univ.\ Math.\ J.} \textbf{61} (2012), no.\ 3, 1129--1156.

\bibitem{Vrzina} M.\ Vržina.
\newblock Cylinders as left invariant CMC surfaces in $\mathrm{Sol}_3$ and $E(\kappa,\tau)$-spaces diffeomorphic to $\mathbb{R}^3$.
\newblock \emph{Diff. Geom. Appl.} \textbf{58} (2018), 141--176.

\end{thebibliography}
\end{document}